\documentclass[a4paper,10pt]{amsart}
\usepackage{enumerate, amsmath, amsfonts, amssymb, amsthm, mathtools, thmtools, wasysym, graphics, graphicx, xcolor, frcursive,xparse,comment,ytableau,stmaryrd,bbm,array,colortbl}
\usepackage[all]{xy}
\usepackage{hyperref}

\usepackage{url, hypcap}
\hypersetup{colorlinks=true, citecolor=darkblue, linkcolor=darkblue}
\usepackage[cal=boondoxo]{mathalfa}
\usepackage{tikz}
\usetikzlibrary{calc,through,backgrounds,shapes,matrix}

\definecolor{darkblue}{rgb}{0.0,0,0.7} 
\newcommand{\darkblue}{\color{darkblue}} 
\definecolor{darkred}{rgb}{0.7,0,0} 
\definecolor{lightgrey}{rgb}{0.7,0.7,0.7} 

\definecolor{meet}{RGB}{255,205,111}
\definecolor{join}{RGB}{0,77,178}

\newtheorem{theorem}{Theorem}[section]
\newtheorem{proposition}[theorem]{Proposition}
\newtheorem{corollary}[theorem]{Corollary}
\newtheorem{lemma}[theorem]{Lemma}

\theoremstyle{definition}
\newtheorem{definition}[theorem]{Definition}
\newtheorem{example}[theorem]{Example}

\newtheorem{remark}[theorem]{Remark}
\usepackage[procnames]{listings}
\usepackage[nameinlink]{cleveref}
\usepackage[all]{xy}
\usepackage[T1]{fontenc}
\crefformat{footnote}{#2\footnotemark[#1]#3}
\crefformat{conjecture}{Conjecture~#2#1#3}

\usepackage[colorinlistoftodos]{todonotes}

\newcommand{\nathan}[1]{\todo[size=\tiny,color=green!30]{#1 \\ \hfill --- N.}}

\newcommand{\defn}[1]{\emph{\darkblue #1}}

\newcommand{\Po}{\mathcal{P}}
\newcommand{\Q}{\mathcal{Q}}
\newcommand{\LL}{\mathcal{L}}
\newcommand{\la}{\mathcal{l}}
\newcommand{\JJ}{\mathcal{J}}
\newcommand{\MM}{\mathcal{M}}
\newcommand{\x}{x}

\newcommand{\y}{y}
\newcommand{\row}{\mathrm{row}}
\newcommand{\down}{\mathrm{D}^\gamma}
\newcommand{\up}{\mathrm{U}^\gamma}
\newcommand{\tog}{\mathrm{flip}^\gamma}
\newcommand{\Camb}{\mathrm{Camb}}
\newcommand{\Krew}{\mathrm{Krew}}
\renewcommand{\mod}{\operatorname{mod}}
\newcommand{\Hom}{\operatorname{Hom}}

\title{Rowmotion in slow motion}

\author[H.~Thomas]{Hugh Thomas}
\address[H.~Thomas]{LaCIM, Universit\'e du Qu\'ebec \`a Montr\'eal}
\email{hugh.ross.thomas@gmail.com}

\author[N.~Williams]{Nathan Williams}
\address[N.~Williams]{University of Texas at Dallas}
\email{nathan.f.williams@gmail.com}

\date{\today}

\begin{document}


\begin{abstract}
Rowmotion is a simple cyclic action on the distributive lattice of order ideals of a poset: it sends the order ideal $x$ to the order ideal generated by the minimal elements not in $x$.  It can also be computed in ``slow motion'' as a sequence of local moves.   We use the setting of trim lattices to generalize both definitions of rowmotion, proving many structural results along the way.  We introduce a flag simplicial complex (similar to the canonical join complex of a semidistributive lattice), and relate our results to recent work of Barnard by proving that extremal semidistributive lattices are trim.  As a corollary, we prove that if $A$ is a representation finite algebra and $\mod A$ has no cycles, then the torsion classes of $A$ ordered by inclusion form a trim lattice.
\end{abstract}

\maketitle

\section{Introduction}





Write $\mathrm{cov}(\Po)$ for the set of cover relations of a poset $\Po$.  Let $\la$ be a second poset, whose elements we call \defn{labels}.  We say that $\gamma: \mathrm{cov}(\Po) \to \la$ is a \defn{labelling} of $\Po$ if any two cover relations involving the same element of $\Po$ are assigned different labels.  Then the set of \defn{downward labels} of an element $\y \in \Po$ is \[\down(\y):=\left\{\gamma(\x \lessdot \y) : \text{ for all } \x \text{ such that } \x \lessdot \y\right\},\]
and its set of \defn{upward labels} is
\[\up(\y) := \left\{\gamma(\y \lessdot z) : \text{ for all } z \text{ such that } \y \lessdot z\right\}.\]

 We say that a labelling $\gamma: \mathrm{cov}(\Po) \to \la$ is \defn{descriptive} if
each element $x \in \Po$ is determined by $\down(\x)$ and also by $\up(\x)$, and if \[\big\{\down(x):x \in \Po\big\}=\big\{\up(x) : x \in \Po\big\}.\]  
When $\gamma$ is descriptive, we define \defn{rowmotion} of $x \in \Po$ by \begin{equation}\row^\gamma(\x):=\text{the unique element } \y \in \Po \text{ such that } \down(\x)=\up(\y).\label{eq:global_row}\end{equation}  

We are interested in a second, slower, way to compute rowmotion.  For any labelling, a \defn{flip} at the label $i \in \la$ on the element $\y \in \Po$ replaces $\y$ by the other endpoint of the unique edge incident to $\y$ labeled by $i$, should such an edge exist.  That is, $\tog_i: \Po \to \Po$ is defined by \[\tog_i (\y) := \begin{cases} \x & \text{if } \gamma(\x \lessdot \y) = i,\\ z & \text{if } \gamma(\y \lessdot z) = i, \\ \y & \text{ otherwise.} \end{cases}\]

We say that rowmotion on a poset with a descriptive labelling can be \defn{computed in slow motion} if we have that $\row^\gamma = \prod_{i \in \la} \tog_{i} $, where the product is taken in the order of any linear extension of $\la$---that is, if rowmotion can be computed as a walk on the Hasse diagram of $\Po$, with the order of the steps respecting the partial order on the labels.

\medskip
We now review two existing examples of rowmotion that can be computed in slow motion.  We will then find a common generalization.


\subsection{Motivation: rowmotion on distributive lattices}

\label{sec:dist_lattice_intro}
Let $\Q$ be a finite poset and write $J(\Q)$ for its distributive lattice of order ideals.   We label a cover relation in $J(\Q)$ by an element of $\Q$ in the following way: \begin{align*}\gamma: \mathrm{cov}(J(\Q)) &\to \Q \\ \gamma(\x \lessdot \y) &:= q  \text{ if } \y = \x \cup \{q\}.\end{align*}  
  Note that $\x$ is
determined by $\down(\x)$ and also by $\up(\x)$: $\down(\x)$ is the antichain in $\Q$ containing the maximal elements of $\x$, and $\up(\x)$ is the antichain containing the minimal elements in $\Q \setminus x$.  Thus, $\gamma$ is descriptive, and rowmotion is defined by~\Cref{eq:global_row} (see~\Cref{sec:examples} for a brief history of its study).  An example of a distributive lattice and its orbits under rowmotion is given in~\Cref{fig:dist_lattice}.




There is a second description of the permutation $\row^\gamma$ as a walk on the Hasse diagram of $J(\Q)$.  For any linear extension $\mathbf{q}$ of $\Q$, rowmotion may be computed as the composition of flips in the ordering on $\Q$ given by $\mathbf{q}$.  Since flips at distinct elements $\x,\y \in \Q$ commute unless $\x \lessdot \y$ or $\y \lessdot \x$ in $\Q$, this definition does not actually depend on the initial choice of linear extension.  

\begin{theorem}[{\cite[Lemma 1]{cameron1995orbits}}]
For any poset $\Q$, rowmotion can be computed in slow motion on the distributive lattice of order ideals $J(\Q)$.
\label{thm:dist}
\end{theorem}

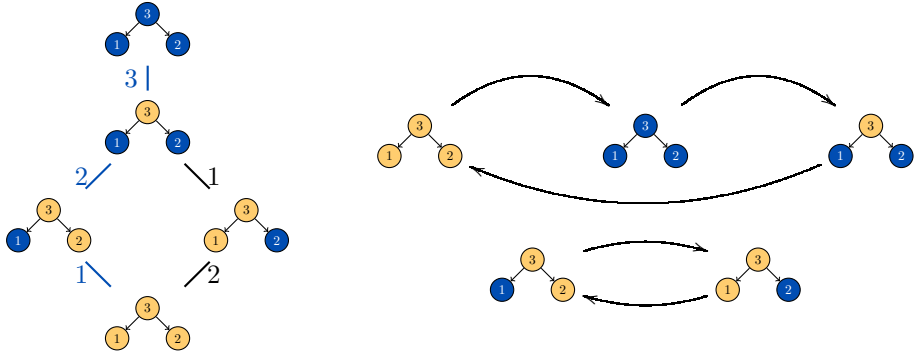
\begin{figure}[htbp]
\raisebox{-0.6\height}{\begin{tikzpicture}[scale=1.3]
\node (x0) [fill=white!20] at (0,0) {\scalebox{.5}{\begin{tikzpicture}[scale=.8] \node (a) [draw=black,circle,thick,fill=meet] at (0,0) {$1$}; \node (b) [draw=black,circle,thick,fill=meet] at (2,0) {$2$}; \node (c) [draw=black,circle,thick,fill=meet] at (1,1) {$3$}; \draw[<-,thick] (a) to (c); \draw[<-,thick] (b) to (c); \end{tikzpicture}}};
\node (x1) [fill=white!20] at (-1,1) {\scalebox{.5}{\begin{tikzpicture}[scale=.8] \node (a) [draw=black,circle,thick,fill=join,text=white] at (0,0) {$1$}; \node (b) [draw=black,circle,thick,fill=meet] at (2,0) {$2$}; \node (c) [draw=black,circle,thick,fill=meet] at (1,1) {$3$}; \draw[<-,thick] (a) to (c); \draw[<-,thick] (b) to (c); \end{tikzpicture}}};
\node (j2) [fill=white!20] at (1,1) {\scalebox{.5}{\begin{tikzpicture}[scale=.8] \node (a) [draw=black,circle,thick,fill=meet] at (0,0) {$1$}; \node (b) [draw=black,circle,thick,fill=join,text=white] at (2,0) {$2$}; \node (c) [draw=black,circle,thick,fill=meet] at (1,1) {$3$}; \draw[<-,thick] (a) to (c); \draw[<-,thick] (b) to (c); \end{tikzpicture}}};
\node (x2) [fill=white!20] at (0,2) {\scalebox{.5}{\begin{tikzpicture}[scale=.8] \node (a) [draw=black,circle,thick,fill=join,text=white] at (0,0) {$1$}; \node (b) [draw=black,circle,thick,fill=join,text=white] at (2,0) {$2$}; \node (c) [draw=black,circle,thick,fill=meet] at (1,1) {$3$}; \draw[<-,thick] (a) to (c); \draw[<-,thick] (b) to (c); \end{tikzpicture}}};
\node (x3) [fill=white!20] at (0,3) {\scalebox{.5}{\begin{tikzpicture}[scale=.8] \node (a) [draw=black,circle,thick,fill=join,text=white] at (0,0) {$1$}; \node (b) [draw=black,circle,thick,fill=join,text=white] at (2,0) {$2$}; \node (c) [draw=black,circle,thick,fill=join,text=white] at (1,1) {$3$}; \draw[<-,thick] (a) to (c); \draw[<-,thick] (b) to (c); \end{tikzpicture}}};
\draw[-,thick,color=join] (x0) to node[midway, left] {$1$} (x1) to node[midway, left] {$2$} (x2) to node[midway, left] {$3$} (x3);
\draw[-,thick] (x0) to node[midway, right] {$2$} (j2) to node[midway, right] {$1$} (x2);
\end{tikzpicture}}\hspace{2em}
\xymatrix@C=1ex{
\raisebox{-0.4\height}{	\scalebox{.5}{\begin{tikzpicture}[scale=.8] \node (a) [draw=black,circle,thick,fill=meet] at (0,0) {$1$}; \node (b) [draw=black,circle,thick,fill=meet] at (2,0) {$2$}; \node (c) [draw=black,circle,thick,fill=meet] at (1,1) {$3$}; \draw[<-,thick] (a) to (c); \draw[<-,thick] (b) to (c); \end{tikzpicture}}} \ar@/^2pc/[rr] & & \raisebox{-0.4\height}{\scalebox{.5}{\begin{tikzpicture}[scale=.8]  \node (a) [draw=black,circle,thick,fill=join,text=white] at (0,0) {$1$}; \node (b) [draw=black,circle,thick,fill=join,text=white] at (2,0) {$2$}; \node (c) [draw=black,circle,thick,fill=join,text=white] at (1,1) {$3$}; \draw[<-,thick] (a) to (c); \draw[<-,thick] (b) to (c); \end{tikzpicture}}} \ar@/^2pc/[rr] & & \raisebox{-0.4\height}{\scalebox{.5}{\begin{tikzpicture}[scale=.8]  \node (a) [draw=black,circle,thick,fill=join,text=white] at (0,0) {$1$}; \node (b) [draw=black,circle,thick,fill=join,text=white] at (2,0) {$2$}; \node (c) [draw=black,circle,thick,fill=meet] at (1,1) {$3$}; \draw[<-,thick] (a) to (c); \draw[<-,thick] (b) to (c); \end{tikzpicture}}} \ar@/^2pc/[llll] \\
		& \raisebox{-0.4\height}{\scalebox{.5}{\begin{tikzpicture}[scale=.8] \node (a) [draw=black,circle,thick,fill=join,text=white] at (0,0) {$1$}; \node (b) [draw=black,circle,thick,fill=meet] at (2,0) {$2$}; \node (c) [draw=black,circle,thick,fill=meet] at (1,1) {$3$}; \draw[<-,thick] (a) to (c); \draw[<-,thick] (b) to (c); \end{tikzpicture}}} \ar@/^1pc/[rr] & & \raisebox{-0.4\height}{\scalebox{.5}{\begin{tikzpicture}[scale=.8]  \node (a) [draw=black,circle,thick,fill=meet] at (0,0) {$1$}; \node (b) [draw=black,circle,thick,fill=join,text=white] at (2,0) {$2$}; \node (c) [draw=black,circle,thick,fill=meet] at (1,1) {$3$}; \draw[<-,thick] (a) to (c); \draw[<-,thick] (b) to (c); \end{tikzpicture}}} \ar@/^1pc/[ll]
}
\caption[On the left is the distributive lattice of order ideals of the root poset of type $A_2$.  On the right are its orbits under rowmotion.]{On the left is the distributive lattice of order ideals of the poset \raisebox{-0.4\height}{\scalebox{.5}{\begin{tikzpicture}[scale=.8]  \node (a) [draw=black,circle,thick] at (0,0) {$1$}; \node (b) [draw=black,circle,thick] at (2,0) {$2$}; \node (c) [draw=black,circle,thick] at (1,1) {$3$}; \draw[<-,thick] (a) to (c); \draw[<-,thick] (b) to (c); \end{tikzpicture}}}.  The maximal chain corresponding to adding the vertices of the poset in the order given by their labelling is indicated in blue.  On the right are the orbits of order ideals under rowmotion.}
\label{fig:dist_lattice}
\end{figure}


\subsection{Motivation: The Kreweras complement on Cambrian lattices}
\label{sec:camb_lattice_intro}
The finite Cambrian lattices are a generalization of the well-known Tamari lattices to finite Coxeter groups.
For a finite Coxeter group $W$ with reflections $T$ and Coxeter element $c$,
let $\mathrm{Sort}(W,c)$ denote the $c$-sortable elements.  They form a sublattice of weak order on $W$.  This lattice can also be realized as a quotient of weak order via the projection $\pi_\downarrow^c:W \to \mathrm{Sort}(W,c)$ that sends an element in $W$ to the largest $c$-sortable element below it~\cite{reading2007sortable}.

In $\Camb_c(W)$, we recall that each reflection $t \in T$ occurs as the cover reflection of a unique join-irreducible $c$-sortable element $j_t$, and it is natural to label cover relations by  \begin{align*}\gamma:\mathrm{cov}(\Camb_c(W))&\to T \\ \gamma(\pi_\downarrow^c(\y s) \lessdot \y) &:= t  \text{ if } \y s = t \y.\end{align*}
 (We are thinking of $T$ as ordered by the heap for the $c$-sorting word for the longest element of $W$.)  It follows from Reading's uniform bijection between $c$-sortable elements and $c$-noncrossing partitions~\cite{reading2007clusters} that $\x$ is
determined by $\down(\x)$ and also by $\up(\x)$: $\down(\x)$ encodes a factorization of a $c$-noncrossing partition $\pi$, and $\up(\x)$ similarly encodes its \defn{Kreweras complement}, the noncrossing partition \[\Krew_c(\pi) := \pi^{-1} c.\]
Then $\gamma$ is descriptive and the rowmotion defined by~\Cref{eq:global_row} coincides with the Kreweras complement: it is a cyclic action of order $2h$, where $h$ is the Coxeter number of $W$.  An example of a Cambrian lattice of type $A_2$ and its orbits under the Kreweras complement is given in~\Cref{fig:camb_lattice}.

By analogy with rowmotion on distributive lattices, the authors and Stump discovered a second description of $\Krew_c$~\cite{williams2013cataland,stump2015cataland}.  While the maximal chains of a distributive lattice $J(\Q)$ are in bijection with the linear extensions of its underlying poset $\Q$, the chains of maximal length in $\Camb_c(W)$ are in bijection with the linear extensions of $T$ (thought of as the heap for the $c$-sorting word for the longest element).  It is natural to consider the composition of flips in the order given by a linear extension of $T$---and any choice recovers the Kreweras complement. 

\begin{theorem}[{\cite[Theorem 4.2.7]{williams2013cataland}}]
For $W$ a finite Coxeter group and $c$ a Coxeter element, the Kreweras complement can be computed in slow motion on the Cambrian lattice $\Camb_c(W)$.
\label{thm:camb}
\end{theorem}

\begin{figure}[htbp]
\scalebox{.9}{\raisebox{-0.6\height}{\begin{tikzpicture}[scale=2]
\node (x0) [align=center] at (0,0) {\scalebox{.5}{\begin{tikzpicture}[scale=.8] \node (a) [draw=black,circle,thick,fill=meet] at (0,0) {$1$}; \node (b) [draw=black,circle,thick,fill=meet] at (2,0) {$3$}; \node (c) [draw=black,circle,thick,fill=meet] at (1,1) {$2$}; \draw[<-,thick] (a) to (c); \draw[->,thick] (b) to (c); \end{tikzpicture}} \\ $e$};
\node (x1) [align=center]at (-1,1) {\scalebox{.5}{\begin{tikzpicture}[scale=.8] \node (a) [draw=black,circle,thick,fill=join,text=white] at (0,0) {$1$}; \node (b) [draw=black,circle,thick,fill=meet] at (2,0) {$3$}; \node (c) [draw=black,circle,thick,fill=meet] at (1,1) {$2$}; \draw[<-,thick] (a) to (c); \draw[->,thick] (b) to (c); \end{tikzpicture}} \\ $s_1$};
\node (j2) [align=center] at (1,1) {\scalebox{.5}{\begin{tikzpicture}[scale=.8] \node (a) [draw=black,circle,thick,fill=meet] at (0,0) {$1$}; \node (b) [draw=black,circle,thick,fill=join,text=white] at (2,0) {$3$}; \node (c) [draw=black,circle,thick,fill=white] at (1,1) {$2$}; \draw[<-,thick] (a) to (c); \draw[->,thick] (b) to (c); \end{tikzpicture}} \\ $s_2$};
\node (x2) [align=center] at (-1,2) {$s_1s_2$ \\ \scalebox{.5}{\begin{tikzpicture}[scale=.8] \node (a) [draw=black,circle,thick,fill=join,text=white] at (0,0) {$1$}; \node (b) [draw=black,circle,thick,fill=meet] at (2,0) {$3$}; \node (c) [draw=black,circle,thick,fill=join,text=white] at (1,1) {$2$}; \draw[<-,thick] (a) to (c); \draw[->,thick] (b) to (c); \end{tikzpicture}} };
\node (x3) [align=center] at (0,3) {$s_1s_2s_1$ \\ \scalebox{.5}{\begin{tikzpicture}[scale=.8] \node (a) [draw=black,circle,thick,fill=join,text=white] at (0,0) {$1$}; \node (b) [draw=black,circle,thick,fill=join,text=white] at (2,0) {$3$}; \node (c) [draw=black,circle,thick,fill=join,text=white] at (1,1) {$2$}; \draw[<-,thick] (a) to (c); \draw[->,thick] (b) to (c); \end{tikzpicture}}};
\draw[-,thick,color=join,line width=2pt] (x0) to node[midway, left] {$1$} (x1) to node[midway, left] {$2$} (x2) to node[midway, left] {$3$} (x3);
\draw[-,thick] (x0) to node[midway, right] {$3$} (j2) to node[midway, right] {$1$} (x3);
\end{tikzpicture}}} \hspace{1em}
\xymatrix@C=1ex{
\raisebox{-0.4\height}{\scalebox{.5}{\begin{tikzpicture}[scale=.8] \node (a) [draw=black,circle,thick,fill=join,text=white] at (0,0) {$1$}; \node (b) [draw=black,circle,thick,fill=meet] at (2,0) {$3$}; \node (c) [draw=black,circle,thick,fill=meet] at (1,1) {$2$}; \draw[<-,thick] (a) to (c); \draw[->,thick] (b) to (c); \end{tikzpicture}}}  \ar@/^2pc/[rr] & & \raisebox{-0.4\height}{\scalebox{.5}{\begin{tikzpicture}[scale=.8] \node (a) [draw=black,circle,thick,fill=meet] at (0,0) {$1$}; \node (b) [draw=black,circle,thick,fill=join,text=white] at (2,0) {$3$}; \node (c) [draw=black,circle,thick,fill=white] at (1,1) {$2$}; \draw[<-,thick] (a) to (c); \draw[->,thick] (b) to (c); \end{tikzpicture}}} \ar@/^2pc/[rr]  & & \raisebox{-0.4\height}{\scalebox{.5}{\begin{tikzpicture}[scale=.8] \node (a) [draw=black,circle,thick,fill=join,text=white] at (0,0) {$1$}; \node (b) [draw=black,circle,thick,fill=meet] at (2,0) {$3$}; \node (c) [draw=black,circle,thick,fill=join,text=white] at (1,1) {$2$}; \draw[<-,thick] (a) to (c); \draw[->,thick] (b) to (c); \end{tikzpicture}}} \ar@/^2pc/[llll] \\
& \raisebox{-0.4\height}{\scalebox{.5}{\begin{tikzpicture}[scale=.8] \node (a) [draw=black,circle,thick,fill=meet] at (0,0) {$1$}; \node (b) [draw=black,circle,thick,fill=meet] at (2,0) {$3$}; \node (c) [draw=black,circle,thick,fill=meet] at (1,1) {$2$}; \draw[<-,thick] (a) to (c); \draw[->,thick] (b) to (c); \end{tikzpicture}}}  \ar@/^1pc/[rr]  &  & \raisebox{-0.4\height}{\scalebox{.5}{\begin{tikzpicture}[scale=.8] \node (a) [draw=black,circle,thick,fill=join,text=white] at (0,0) {$1$}; \node (b) [draw=black,circle,thick,fill=join,text=white] at (2,0) {$3$}; \node (c) [draw=black,circle,thick,fill=join,text=white] at (1,1) {$2$}; \draw[<-,thick] (a) to (c); \draw[->,thick] (b) to (c); \end{tikzpicture}}}  \ar@/^1pc/[ll]
}
\caption{On the left is the $c$-Cambrian lattice $\Camb_{c}(A_2)$ for $c=s_1s_2$, which recovers the Tamari lattice on five elements.  The maximal chain corresponding to the labelling of the vertices of the poset is indicated in blue, and elements are labeled by their $c$-sorting word, as well as by their inversion set (in blue) and the complement of the inversion set of the corresponding $c$-antisortable element (in orange).  On the right are the orbits of $c$-sortable elements under the Kreweras complement.   The similarities with~\Cref{fig:dist_lattice} are not a coincidence---see~\Cref{rem:same_number} and \cite{panyushev2009orbits,armstrong2013uniform}.}
\label{fig:camb_lattice}
\end{figure}
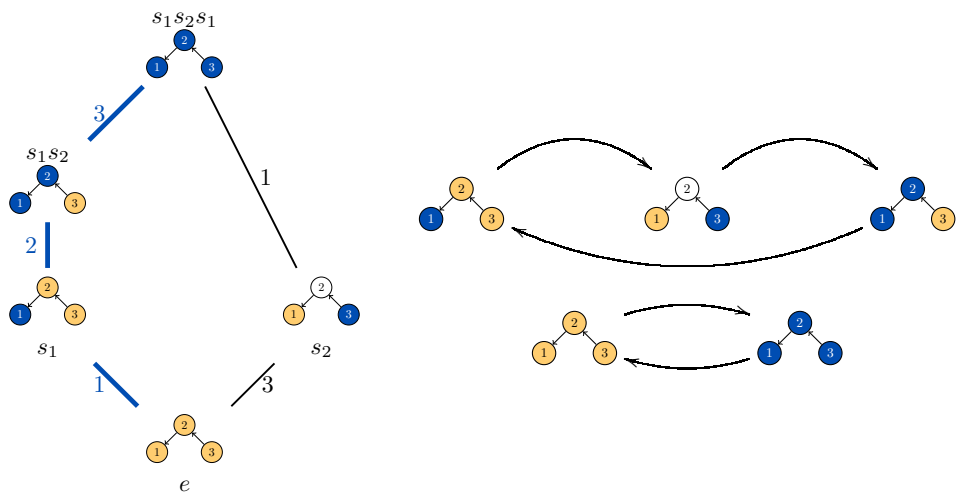


\subsection{Unification: trim lattices}The similarity of the situations of~\Cref{sec:dist_lattice_intro,sec:camb_lattice_intro} suggests a unification.
A \defn{trim lattice} is an extremal left modular lattice (see~\Cref{def:extremal,def:left_modular}).  Trim lattices were introduced by the first author as an analogue of distributive lattices that do not require the lattice to be graded~\cite{thomas2006analogue}.    In particular, distributive lattices are exactly the graded trim lattices~\cite[Theorem 2]{thomas2006analogue}; the finite Cambrian lattices of~\Cref{sec:camb_lattice_intro} are also trim~\cite{thomas2006analogue,IT,muhle2016trimness}; we prove in~\Cref{cor:rep_are_trim} that trim lattices arise naturally in the representation theory of finite-dimensional algebras as the lattices of torsion classes for suitably chosen algebras.

Our main theorem identifies trim lattices as a setting for a common generalization of~\Cref{thm:dist,thm:camb}:

\begin{theorem}
    Any trim lattice has a descriptive labelling, and rowmotion can be computed in slow motion.
\label{thm:main_thm}
\end{theorem}

\subsection{Rowmotion on semidistributive lattices}
To generalize the scope of rowmotion futher, it is natural to search for other classes of posets with descriptive labellings.  In this direction, Barnard recently showed that semidistributive lattices have descriptive labellings using the bijection between join- and meet-irreducibles of a semidistributive lattice~\cite{freese1995free,day1979characterizations}, along with the natural labelling of the cover relations by these irreducibles~\cite{barnard2016canonical}.

As the join- and meet-irreducibles of a semidistributive lattice do not usually come endowed with a natural linear order, it is not clear how one should compute rowmotion in slow motion.  However, if the length of the longest chain (which is at most the number of join-irreducibles or the number of meet-irreducibles) should equal this extremal value, then the descriptive labelling of a maximal-length chain induces a linear order on the labels.  (Lattices for which these three values are equal are called \defn{extremal lattices}.) The following general theorem implies that flipping in this order agrees with rowmotion.

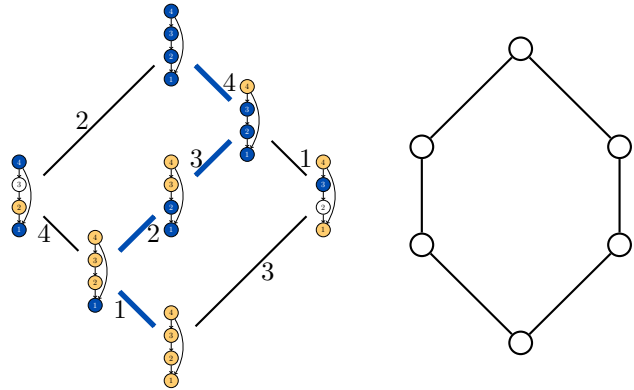
\begin{figure}[htbp]
\raisebox{-0.5\height}{\begin{tikzpicture}[scale=1]
\node (a) [align=center] at (0,0) {\scalebox{0.3}{\begin{tikzpicture}
\node (1) [circle,thick,draw,fill=meet] at (0,0) {1};
\node (2) [circle,thick,draw,fill=meet] at (0,1) {2};
\node (3) [circle,thick,draw,fill=meet] at (0,2) {3};
\node (4) [circle,thick,draw,fill=meet] at (0,3) {4};
\draw[->,thick] (2) to (1);
\draw[->,thick] (3) to (2);
\draw[->,thick] (4) to (3);
\draw[->,thick] (4) to [bend left]  (1);
\end{tikzpicture}}};
\node (b)  at (-1,1) {\scalebox{0.3}{\begin{tikzpicture}
\node (1) [circle,thick,draw,fill=join,text=white] at (0,0) {1};
\node (2) [circle,thick,draw,fill=meet] at (0,1) {2};
\node (3) [circle,thick,draw,fill=meet] at (0,2) {3};
\node (4) [circle,thick,draw,fill=meet] at (0,3) {4};
\draw[->,thick] (2) to (1);
\draw[->,thick] (3) to (2);
\draw[->,thick] (4) to (3);
\draw[->,thick] (4) to [bend left]  (1);
\end{tikzpicture}}};
\node (c) at (-2,2) {\scalebox{0.3}{\begin{tikzpicture}
\node (1) [circle,thick,draw,fill=join,text=white] at (0,0) {1};
\node (2) [circle,thick,draw,fill=meet] at (0,1) {2};
\node (3) [circle,thick,draw,fill=white] at (0,2) {3};
\node (4) [circle,thick,draw,fill=join,text=white] at (0,3) {4};
\draw[->,thick] (2) to (1);
\draw[->,thick] (3) to (2);
\draw[->,thick] (4) to (3);
\draw[->,thick] (4) to [bend left]  (1);
\end{tikzpicture}}};
\node (d) at (0,2) {\scalebox{0.3}{\begin{tikzpicture}
\node (1) [circle,thick,draw,fill=join,text=white] at (0,0) {1};
\node (2) [circle,thick,draw,fill=join,text=white] at (0,1) {2};
\node (3) [circle,thick,draw,fill=meet] at (0,2) {3};
\node (4) [circle,thick,draw,fill=meet] at (0,3) {4};
\draw[->,thick] (2) to (1);
\draw[->,thick] (3) to (2);
\draw[->,thick] (4) to (3);
\draw[->,thick] (4) to [bend left]  (1);
\end{tikzpicture}}};
\node (e)  at (2,2) {\scalebox{0.3}{\begin{tikzpicture}
\node (1) [circle,thick,draw,fill=meet] at (0,0) {1};
\node (2) [circle,thick,draw,fill=white] at (0,1) {2};
\node (3) [circle,thick,draw,fill=join,text=white] at (0,2) {3};
\node (4) [circle,thick,draw,fill=meet] at (0,3) {4};
\draw[->,thick] (2) to (1);
\draw[->,thick] (3) to (2);
\draw[->,thick] (4) to (3);
\draw[->,thick] (4) to [bend left]  (1);
\end{tikzpicture}}};
\node (f) at (1,3) {\scalebox{0.3}{\begin{tikzpicture}
\node (1) [circle,thick,draw,fill=join,text=white] at (0,0) {1};
\node (2) [circle,thick,draw,fill=join,text=white] at (0,1) {2};
\node (3) [circle,thick,draw,fill=join,text=white] at (0,2) {3};
\node (4) [circle,thick,draw,fill=meet] at (0,3) {4};
\draw[->,thick] (2) to (1);
\draw[->,thick] (3) to (2);
\draw[->,thick] (4) to (3);
\draw[->,thick] (4) to [bend left]  (1);
\end{tikzpicture}}};
\node (g) at (0,4) {\scalebox{0.3}{\begin{tikzpicture}
\node (1) [circle,thick,draw,fill=join,text=white] at (0,0) {1};
\node (2) [circle,thick,draw,fill=join,text=white] at (0,1) {2};
\node (3) [circle,thick,draw,fill=join,text=white] at (0,2) {3};
\node (4) [circle,thick,draw,fill=join,text=white] at (0,3) {4};
\draw[->,thick] (2) to (1);
\draw[->,thick] (3) to (2);
\draw[->,thick] (4) to (3);
\draw[->,thick] (4) to [bend left]  (1);
\end{tikzpicture}}};
\draw[-,thick] (a) to  node[midway, left] {$1$} (b) to  node[midway, left] {$4$} (c) to  node[midway, left] {$2$} (g);
\draw[-,thick] (a) to  node[midway, right] {$3$} (e) to  node[midway, right] {$1$} (f) to  node[midway, right] {$4$} (g);
\draw[-,thick] (b)  to node[midway, right] {$2$} (d) to node[midway, left] {$3$} (f);
\draw[-,thick,color=join,line width=2pt] (a) to (b) to  (d) to (f) to (g);
\end{tikzpicture}}\hspace{2em}
\raisebox{-0.5\height}{\begin{tikzpicture}[scale=1.3]
\node (a) [circle,thick,minimum size=3pt,inner sep=3pt,draw,fill=white] at (0,0) {};
\node (b) [circle,thick,minimum size=3pt,inner sep=3pt,draw,fill=white]  at (-1,1) {};
\node (c) [circle,thick,minimum size=3pt,inner sep=3pt,draw,fill=white] at (-1,2) {};
\node (d) [circle,thick,minimum size=3pt,inner sep=3pt,draw,fill=white]  at (0,3) {};
\node (e) [circle,thick,minimum size=3pt,inner sep=3pt,draw,fill=white]  at (1,1) {};
\node (f) [circle,thick,minimum size=3pt,inner sep=3pt,draw,fill=white]  at (1,2) {};
\draw[-,thick] (a) to (b) to (c) to (d);
\draw[-,thick] (a) to (e) to (f) to (d);
\end{tikzpicture}}
\caption{On the left is a trim lattice that is a minimal counterexample to semidistributivity~\cite{davey1975characterization}.  On the right is a semidistributive lattice that is neither left modular nor extremal; for this lattice, there is no linear order on the join-irreducibles that allows us to compute rowmotion in slow motion.}
\label{fig:trim_semi}
\end{figure}

\begin{theorem}
Extremal semidistributive lattices are trim.  (By~\Cref{thm:main_thm}, rowmotion can therefore be computed in slow motion.)
\label{thm:barn}
\end{theorem}

Our results are complementary to Barnard's: there are trim lattices which are not semidistributive (such as the lattice on the left of \Cref{fig:trim_semi}), and there are also semidistributive lattices which are neither left modular nor extremal (such as the lattice on the right of \Cref{fig:trim_semi}).

\Cref{thm:barn} also has a useful application to the representation theory of finite-dimensional algebras.

\begin{corollary}
\label{cor:rep_are_trim}
Let $A$ be a finite-dimensional algebra over a field $k$.  If $A$ is representation finite and $\mod A$ has no cycles, then the torsion pairs of $A$, ordered with respect to inclusion of torsion classes (or reverse inclusion of torsion-free classes), form a trim lattice.
\end{corollary}

\subsection{Organization}
The remainder of the paper is organized as follows.  In~\Cref{sec:lattices}, we review distributive, extremal, and left modular lattices, recalling representation theorems using Galois graphs and labellings of cover relations.  We show that examples of extremal lattices arise naturally in the study of finite-dimensional algebras.  We turn to trim lattices in~\Cref{sec:trim}: we summarize their structural properties in~\Cref{sec:structure}, characterize their covers in terms of Galois graphs, and introduce a new recursive decomposition in~\Cref{sec:decomposition} that generalizes the Cambrian recurrence on Cambrian lattices.  In~\Cref{sec:global_row}, we prove~\Cref{thm:main_thm}.  In \Cref{sec:labelcomplex}, motivated by the canonical join complex of a semidistributive lattice, we define the independence complex of a trim lattice as the complex of downward labels.  We show in~\Cref{prop:ind_complex} that the independence complex is a flag simplicial complex, and in~\Cref{cor:independent} that it coincides with the simplicial complex of independent sets of the Galois graph.  In~\Cref{sec:semidistributive}, we recall Barnard's constructions on semidistributive lattices, prove~\Cref{thm:barn} and~\Cref{cor:rep_are_trim}, and use the independence complex to relate the canonical join complex of an extremal semidistributive lattice to its Galois graph.  Finally, in~\Cref{sec:examples}, we give a short account of the history of rowmotion and briefly summarize what is known about rowmotion on different families of interesting lattices.

\section{Lattices: Representations and Labellings}
\label{sec:lattices}
{\bf Note: all posets in this paper are assumed to be finite. }

In this section, we review three classes of lattices: distributive lattices---which serve as some of the simplest examples of the theory---as well as extremal  and left modular lattices.  Both extremal and left modular lattices generalize desirable properties of distributive lattices: extremal lattices have a representation theorem similar to the fundamental theorem of finite distributive lattices (and naturally arise in the representation theory of finite-dimensional algebras), while the cover relations of left modular lattices may be naturally labeled by join- or meet-ir\-reducibles.  The happy intersection of these two classes of lattices are the trim lattices, which we will review in~\Cref{sec:trim}.

\subsection{Lattices}
\label{ssec:lattices}
A \defn{lattice} $\LL$ is a nonempty partially ordered set with order relation written $\leq$, such that any two elements $\x,\y \in \LL$ have a \defn{join} (a least upper bound, written $\x \vee \y$) and a \defn{meet} (a greatest lower bound, written $\x \wedge \y$).  A \defn{chain} is a sequence of elements \[\x_0 < \x_1 < \cdots < \x_r,\] of \defn{length} $r$.  The length of $\LL$ is the maximum of the lengths of its chains.   A \defn{join-irreducible} element is an element $\x \in \LL$ such that $\y \vee z \neq \x$ for any $\y,z<\x$.  Similarly, if $\y \wedge z \neq \x$ for any $\y,z > \x$, we call $\x$ \defn{meet-irreducible}.   We write $\JJ$ for the set of all join-irreducible elements of a lattice $\LL$, and $\MM$ for the set of all of its meet-irreducible elements.  We denote the unique minimal element of $\LL$ by $\hat{0}$, and its unique maximal element by $\hat{1}$.  By convention $\hat 0$ does not count as join-irreducible, and $\hat 1$ does not count as meet-irreducible.

A \defn{lattice congruence} on a lattice $\LL$ is an equivalence relation $\equiv$ on $\LL$ such that whenever $x_1 \equiv x_2$ and $y_1 \equiv y_2$ then $x_1\wedge y_1 \equiv x_2\wedge y_2$ and $x_1\vee y_1 \equiv x_2\vee y_2$. The \defn{quotient} of $\LL$ with respect to $\equiv$ is the lattice defined on the congruence classes.

Fix a lattice $\LL$.  In general, computing meets and joins can be nontrivial, and so it is reasonable to ask for a ``representation'' of $\LL$, so that elements of $\LL$ are represented by objects for which joins and meets are easy to compute~\cite{birkhoff1948representations}.  For example, joins and meets in a Boolean lattice are easily computed when the Boolean lattice is represented as the set of subsets of a finite set, ordered by containment.

For $X \subseteq \JJ$ and for $Y \subseteq \MM$, we follow (the dual of)  Markowsky's \defn{poset of irreducibles}~\cite{markowsky1975factorization} to define\footnote{
Similar ideas appeared after~\cite{markowsky1975factorization} in the work of Rudolf Wille under the names ``concept lattices'' and ``contexts'', and have been used in the study of ``Formal Concept Analysis''~\cite{wille1982restructuring}.}
\begin{align*}
\MM(X) &= \{m \in \MM : j \leq m \text{ for all } j \in X\} \text{ and} \\ \JJ(Y)&= \{j \in \JJ : j \leq m \text{ for all } m \in Y\}.\end{align*}

\begin{theorem}[{\cite[Theorem 4]{markowsky1975factorization}}]
\label{thm:set_representation}
A finite lattice $\LL$ is isomorphic to the set of pairs \[\Big\{(X,Y) \subset \JJ \times \MM : \MM(X)=Y \text{ and } \JJ(Y)=X\Big\},\] under inclusion of the first component or under reverse inclusion of the second component.  Moreover, joins and meets in $\LL$ may be computed as \begin{align*}(X,Y) \vee (X',Y') &= (\JJ(Y \cap Y'), Y \cap Y') \text{ and} \\ (X,Y) \wedge (X',Y') &= (X \cap X', \MM(X \cap X')).\end{align*}
\end{theorem}

\subsection{Distributive lattices}
\label{sec:distributive}
The \defn{distributive lattices} are a particularly simple class of lattice that have an elegant representation theorem: Birkhoff's \defn{fundamental
theorem of finite distributive lattices}.

\begin{theorem}
The lattice $J(\Q)$ of order ideals of a finite poset $\Q$ (ordered by containment; meet and join are given by union and intersection) is a distributive lattice, and every finite distributive
lattice arises in this way.
\label{thm:birkhoff}
\end{theorem}

An example of~\Cref{thm:birkhoff} is given in~\Cref{fig:dist_lattice}.  Writing $j_q$ for the join-irreducible in $J(\Q)$ whose unique maximal element is $q$, we label a cover relation $\x \lessdot \y$ in $J(\Q)$ by \[\gamma_\JJ(\x \lessdot \y) := q  \text{ if } \y = \x \vee j_q.\]   Dually, writing $m_q$ for the meet-irreducible whose complement in $\Q$ has $q$ as its unique minimal element, we may also label a cover relation $\x \lessdot \y$ in $J(\Q)$ by \[\gamma_\MM(\x \lessdot \y) := q  \text{ if } \x = \y \wedge m_q.\]   It is clear that these two labellings agree, so that we may define a consistent labelling of the cover relations in $J(\Q)$ \[\gamma:= \gamma_\JJ = \gamma_\MM.\] 

Identifying $\JJ$ and $\MM$ using the labelling $\gamma$,  a specialization of~\Cref{thm:set_representation} recovers Birkhoff's representation theorem. 

\subsection{Extremal lattices}
\label{sec:extremal}

Note that a lattice of length $n$ must have at least $n$ join-irreducible elements and at least $n$ meet-irreducible elements.  Finite distributive lattices are characterized among finite lattices by the property that the length of \emph{any} maximal chain is equal to both the number of join-irreducible elements and to the number of meet-irreducibles.  Markowsky modified this characterization to define extremal lattices by dropping the condition that all maximal chains have equal length.

\begin{definition}[\cite{markowsky1992primes}]
 A lattice of length $n$ with $|\JJ|=|\MM|=n$ is called \defn{extremal}.
\label{def:extremal}
\end{definition}

To state Markowsky's \defn{representation theorem for extremal lattices}, we fix an extremal lattice $\LL$ and a maximal-length chain
\[\hat 0 = x_0 \lessdot x_1 \lessdot \dots \lessdot x_n=\hat 1.\]  We use this chain to index the join-irreducible elements $j_1,j_2,\ldots,j_n$ and the meet-irreducible elements $m_1,m_2,\ldots, m_n$, so that
\[x_i=j_1\vee \dots \vee j_i=m_{i+1}\wedge \dots \wedge m_n.\]

  As with distributive lattices, for extremal lattices we may identify elements of $\JJ$ and elements of $\MM$ with their index induced by the maximal-length chain.  This leads to the following specialization of~\Cref{thm:set_representation}.

  Define the \defn{Galois graph}\footnote{We refer to $G(\LL)$ as the Galois graph because we can view it as defining a Galois connection between subsets of join-irreducibles and subsets of meet-irreducibles, as in \cite{markowsky1975factorization,DP}.} of $\LL$ to be the graph $G(\LL)$ whose vertex set is $[n]:=\{1,2,\ldots,n\}$, with a directed edge $i\rightarrow k$ when $j_i \not\leq m_k$ and $i\neq k$.  The ordering of the join-irreducibles and meet-irreducibles imply that $i \rightarrow k$ only if $i> k$.  The \defn{Galois poset} of $\LL$ is the poset $P(\LL)$ defined by the transitive closure of $G(\LL)$.

Conversely, any directed graph $G$ on $[n]$ with no multiple edges satisfying $i \rightarrow k$ only if $i> k$ gives rise to an extremal lattice $\LL(G)$, as we now explain.  For $X,Y \subseteq [n]$ with $X \cap Y = \emptyset$, we say $(X,Y)$ is an \defn{orthogonal pair} if there is no arrow from any $i\in X$ to any $k\in Y$, and we say it is a \defn{maximal orthogonal pair} if $X$ and $Y$ are maximal with that property.  Clearly, to each $Y\subseteq [n]$, there is at most one $X$ such that $(X,Y)$ is a maximal orthogonal pair (and conversely).  Then the extremal lattice associated to $G$ is equivalently given by \emph{either} of
\begin{align*}
(X,Y) \leq (X',Y') &\text{ iff } X \subseteq X', \text{ or} \\
(X,Y) \leq (X',Y') &\text{ iff } Y' \subseteq Y.
\end{align*}
Furthermore, the join is computed by intersecting the second terms, while meet is given by the intersection of the first terms.

If $\x$ is an element of an extremal lattice $\LL$ with corresponding maximal orthogonal pair $(X,Y)$, we write $x_\JJ=X$ and $x_\MM=Y$.  (Given an extremal lattice, $x_\JJ$ indexes the join-irreducible elements below $x$, while $x_\MM$ indexes the meet-irreducible elements above $x$.)

We summarize our discussion of Markowsky's representation theorem for finite extremal lattices with the following theorem, which is illustrated in~\Cref{fig:dist_lattice,fig:camb_lattice,fig:trim_semi,fig:extremal_no_global,fig:tamari_extremal}.

\begin{theorem}[\cite{markowsky1992primes}]
Every finite extremal lattice $\LL$ is isomorphic to the lattice of  maximal orthogonal pairs of its Galois graph: $\LL \simeq \LL(G(\LL))$.  Conversely, given any directed graph $G$ on $[n]$ with no multiple edges and satisfying $i \rightarrow k$ only if $i> k$, the lattice $\LL(G)$ of maximal orthogonal pairs of $G$ is extremal.
\label{thm:extreme_representation}
\end{theorem}

\begin{example}
We may recover the underlying poset $\Q$ of  a distributive lattice $J(\Q)$ from the Galois graph $G(J(\Q))$ in the following way. Identifying elements of $\Q$ with the labels of the corresponding join- and meet-irreducible elements, the Galois graph of $J(\Q)$ has an edge for each strict relation in $\Q$.  Conversely, if the arrows of a Galois graph define a transitive relation, then the resulting extremal lattice is a distributive lattice.

For distributive lattices, maximal orthogonal pairs $(X,Y)$ have the property that $X \sqcup Y = [n]$: $X$ is an order ideal, and $Y$ is the complementary order filter.  See~\Cref{fig:dist_lattice}.
\label{ex:spine}
\end{example}

The \defn{spine} of an extremal lattice consists of those elements which lie on some chain of maximal length.  Following~\cite[Lemma 6]{thomas2006analogue}, we have the following characterization of the spine, which extends~\Cref{ex:spine} to any extremal lattice.

\begin{proposition}
\label{thm:spinea}
The spine of an extremal lattice $\LL$ is a distributive sublattice of $\LL$.  It is isomorphic to $J(P(\LL))$, the lattice of order ideals of the Galois poset.
\end{proposition}

\begin{proof}
We represent the elements of $\LL$ as maximal orthogonal pairs of $G(\LL)$ using~\Cref{thm:extreme_representation}.  It is clear that a maximal orthogonal pair $(X,Y)$ is in the spine if and only if $X \cup Y = [n]$.  Furthermore, $X \cup Y = [n]$ is equivalent to $X$ being an order ideal in $P(\LL)$.

By definition, $\left((X_1,Y_1) \wedge (X_2,Y_2)\right)_\JJ = X_1 \cap X_2$.  If $(X_1,Y_1)$ and $(X_2,Y_2)$ are in the spine, then $X_1$ and $X_2$ are order ideals in $\Po$, $X_1 \cap X_2$ is also an order ideal, and so $(X_1,Y_1) \wedge (X_2,Y_2)$ is still an element of the spine.  The dual argument shows that the spine is also closed under join.
\end{proof}

\begin{proposition}
\label{cor:galois_unique}
Any two maximal length chains of an extremal lattice $\LL$ induce the same correspondence between $\JJ$ and $\MM$.  A Galois graph $G(\LL)$ is unique up to renumbering; any renumbering is necessarily order-preserving with respect to $P(\LL)$.
\end{proposition}

\begin{proof}
Let $\LL$ be an extremal lattice.  We want to show that choosing different
chains of maximal length for $\LL$ result in Galois graphs which differ only by
a relabelling.

The chains of maximal length for $\LL$ are contained in the spine of
$\LL$, which is a distributive lattice by \Cref{thm:spinea}.  We may assume
that we have two chains of maximal length in $\LL$ which differ in only one
place, since any two chains
of maximal length in a distributive lattice differ by a sequence of such
differences.  Let $(a_i)_{i=0}^n$ and $(b_i)_{i=0}^n$ be the two chains,
and suppose they differ in position $k$.  Let $j_1,\dots,j_n$ be the numbering
of join-irreducibles induced by the chain $a$, and let $j_1',\dots,j_n'$ be
the numbering of join-irreducibles induced by the chain $b$; and similarly for
meet-irreducibles.  Clearly $j'_i=j_i$ for $i\ne k,k+1$.  Note further that
$j'_k \ne j_k$, so it must be the case that $j'_k=j_{k+1}$ and
$j'_{k+1}=j_k$.  The same argument applies to meet-irreducibles.
Thus, the two chains induce the same correspondence
between join-irreducibles and meet-irreducibles.
Let $G(\LL)$ and $G'(\LL)$ be the Galois graphs associated to chains $a$ and
$b$ respectively.  By what we have just shown, $G(\LL)$ and $G'(\LL)$ differ
only by swapping the labels $k$ and $k+1$.  Since all arrows in
$G(\LL)$ go from larger-numbered to smaller-numbered labels, and the same
thing is true after swapping $k$ and $k+1$, it follows that swapping these
two labels preserves the poset $P(\LL)$.
\end{proof}

We also have the following result, which says that the property of being an extremal lattice is
preserved under lattice quotients as defined in~\Cref{ssec:lattices}.
\begin{lemma}\label{lem:ext-quot}
Extremal lattices are preserved under lattice quotients.
\end{lemma}

\begin{proof}
Let $\LL$ be an extremal lattice, with maximal-length chain $(x_i)_{i=0}^n$.  Let $\LL'=\LL/{\equiv}$ be a
lattice quotient of $\LL$.  Recall that the equivalence classes with respect to $\equiv$ are necessarily intervals in $\LL$.  A join-irreducible element of $\LL'$ corresponds to an interval in $\LL$ whose minimum element is join-irreducible.

Suppose $x_{i-1}\equiv x_i$.  Let $j_i\gtrdot (j_i)_*$ be the join-irreducible corresponding to $x_i$ and its unique cover in $\LL$.  Since $(j_i)_*\vee x_{i-1}=x_{i-1}$ but $j_i\vee x_{i-1}=x_i$, and we have $x_{i-1}\equiv x_i$, we must have $(j_i)_*\equiv j_i$.  Thus, $j_i$ is not the minimum element of its equivalence class in $\LL'$.  It follows that the number of join-irreducibles of $\LL'$ is at most the length of $\LL'$, so in fact they are equal.  The same holds for meet-irreducibles, so $\LL'$ is extremal.
\end{proof}

\begin{example}
The Galois graph of a Cambrian lattice may be constructed as follows, as illustrated in~\Cref{fig:camb_lattice,fig:tamari_extremal,fig:b2camb}. (Non-experts may ignore the description below --- \Cref{fig:camb_lattice,fig:tamari_extremal,fig:b2camb} can still be appreciated as examples of extremal lattices without knowing where the Galois graph comes from.)  For $c=a_1 a_2 \cdots a_n$ a permutation of the simple reflections of $W$, Reading's \defn{$c$-sorting word} for the long element of $W$ is the leftmost reduced word for $w_\circ$ in the word $c^\infty$: \[w_\circ(c) = a_{1}a_2 \cdots a_{N} \text{ with each } a_i \in S.\]  The Galois graph then has vertex set $[N]$ (indexing the reflections $T$ of $W$) and a directed edge $j\to i$ iff $j>i$ and $(a_i \cdots a_{j-1}) a_j (a_{j-1} \cdots a_i)$ does not lie in the parabolic subgroup of $W$ generated by all simple reflections except $a_i$.

Cambrian lattices are lattice quotients of weak order.  A $c$-sortable element $w \in \Camb_c(W)$ is the minimal element of a collapsed weak-order interval; denote the corresponding maximal element by $u$. One can check that in the maximal orthogonal pair $(w_\JJ,w_\MM)$, $w_\JJ$ encodes the inversion set of $w$, and $w_\MM$ encodes the inversion set of $w_\circ u$. 
\label{ex:cambrian}
\end{example}

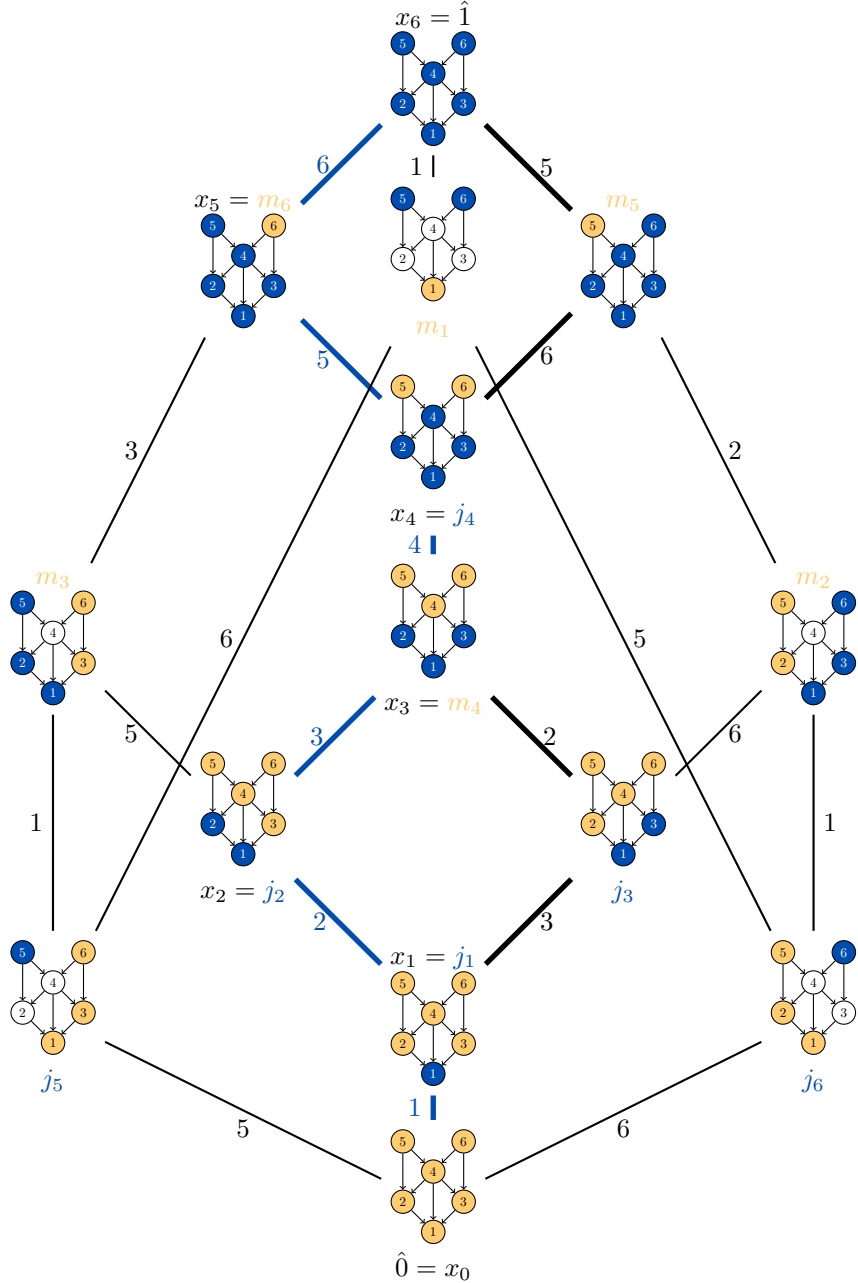
\begin{figure}[htbp]
\begin{tikzpicture}[scale=2.5]
\node (x0) [align=center] at (0,0) {\scalebox{0.5}{\begin{tikzpicture}[scale=.8]
\node (1) [circle,thick,draw,fill=meet] at (0,0) {1};
\node (2) [circle,thick,draw,fill=meet]  at (-1,1) {2};
\node (3) [circle,thick,draw,fill=meet] at (1,1) {3};
\node (4) [circle,thick,draw,fill=meet] at (0,2) {4};
\node (5) [circle,thick,draw,fill=meet] at (-1,3) {5};
\node (6) [circle,thick,draw,fill=meet] at (1,3) {6};
\draw[->,thick] (5) to (4);\draw[->,thick] (5) to (2);\draw[->,thick] (6) to (3);
\draw[->,thick] (4) to (2);
\draw[->,thick] (2) to (1);
\draw[->,thick] (6) to (4);
\draw[->,thick] (4) to (3);
\draw[->,thick] (3) to (1);
\draw[->,thick] (4) to (1);
\end{tikzpicture}} \\ $\hat{0}=x_0$};
\node (x1) [align=center] at (0,1) { $x_1=\textcolor{join}{j_1}$\\ \scalebox{0.5}{\begin{tikzpicture}[scale=.8]
\node (1) [circle,thick,draw,fill=join,text=white] at (0,0) {1};
\node (2) [circle,thick,draw,fill=meet]  at (-1,1) {2};
\node (3) [circle,thick,draw,fill=meet] at (1,1) {3};
\node (4) [circle,thick,draw,fill=meet] at (0,2) {4};
\node (5) [circle,thick,draw,fill=meet] at (-1,3) {5};
\node (6) [circle,thick,draw,fill=meet] at (1,3) {6};
\draw[->,thick] (5) to (4);\draw[->,thick] (5) to (2);\draw[->,thick] (6) to (3);
\draw[->,thick] (4) to (2);
\draw[->,thick] (2) to (1);
\draw[->,thick] (6) to (4);
\draw[->,thick] (4) to (3);
\draw[->,thick] (3) to (1);
\draw[->,thick] (4) to (1);
\end{tikzpicture}}};
\node (x2) [align=center] at (-1,2) {\scalebox{0.5}{\begin{tikzpicture}[scale=.8]
\node (1) [circle,thick,draw,fill=join,text=white] at (0,0) {1};
\node (2) [circle,thick,draw,fill=join,text=white]  at (-1,1) {2};
\node (3) [circle,thick,draw,fill=meet] at (1,1) {3};
\node (4) [circle,thick,draw,fill=meet] at (0,2) {4};
\node (5) [circle,thick,draw,fill=meet] at (-1,3) {5};
\node (6) [circle,thick,draw,fill=meet] at (1,3) {6};
\draw[->,thick] (5) to (4);\draw[->,thick] (5) to (2);\draw[->,thick] (6) to (3);
\draw[->,thick] (4) to (2);
\draw[->,thick] (2) to (1);
\draw[->,thick] (6) to (4);
\draw[->,thick] (4) to (3);
\draw[->,thick] (3) to (1);
\draw[->,thick] (4) to (1);
\end{tikzpicture}} \\ $x_2=\textcolor{join}{j_2}$};
\node (x3) [align=center] at (0,3) {\scalebox{0.5}{\begin{tikzpicture}[scale=.8]
\node (1) [circle,thick,draw,fill=join,text=white] at (0,0) {1};
\node (2) [circle,thick,draw,fill=join,text=white]  at (-1,1) {2};
\node (3) [circle,thick,draw,fill=join,text=white] at (1,1) {3};
\node (4) [circle,thick,draw,fill=meet] at (0,2) {4};
\node (5) [circle,thick,draw,fill=meet] at (-1,3) {5};
\node (6) [circle,thick,draw,fill=meet] at (1,3) {6};
\draw[->,thick] (5) to (4);\draw[->,thick] (5) to (2);\draw[->,thick] (6) to (3);
\draw[->,thick] (4) to (2);
\draw[->,thick] (2) to (1);
\draw[->,thick] (6) to (4);
\draw[->,thick] (4) to (3);
\draw[->,thick] (3) to (1);
\draw[->,thick] (4) to (1);
\end{tikzpicture}} \\ $x_3=\textcolor{meet}{m_4}$};
\node (x4) [align=center] at (0,4) {\scalebox{0.5}{\begin{tikzpicture}[scale=.8]
\node (1) [circle,thick,draw,fill=join,text=white] at (0,0) {1};
\node (2) [circle,thick,draw,fill=join,text=white]  at (-1,1) {2};
\node (3) [circle,thick,draw,fill=join,text=white] at (1,1) {3};
\node (4) [circle,thick,draw,fill=join,text=white] at (0,2) {4};
\node (5) [circle,thick,draw,fill=meet] at (-1,3) {5};
\node (6) [circle,thick,draw,fill=meet] at (1,3) {6};
\draw[->,thick] (5) to (4);\draw[->,thick] (5) to (2);\draw[->,thick] (6) to (3);
\draw[->,thick] (4) to (2);
\draw[->,thick] (2) to (1);
\draw[->,thick] (6) to (4);
\draw[->,thick] (4) to (3);
\draw[->,thick] (3) to (1);
\draw[->,thick] (4) to (1);
\end{tikzpicture}} \\ $x_4=\textcolor{join}{j_4}$};
\node (x5) [align=center] at (-1,5) {$x_5=\textcolor{meet}{m_6}$ \\ \scalebox{0.5}{\begin{tikzpicture}[scale=.8]
\node (1) [circle,thick,draw,fill=join,text=white] at (0,0) {1};
\node (2) [circle,thick,draw,fill=join,text=white]  at (-1,1) {2};
\node (3) [circle,thick,draw,fill=join,text=white] at (1,1) {3};
\node (4) [circle,thick,draw,fill=join,text=white] at (0,2) {4};
\node (5) [circle,thick,draw,fill=join,text=white] at (-1,3) {5};
\node (6) [circle,thick,draw,fill=meet] at (1,3) {6};
\draw[->,thick] (5) to (4);\draw[->,thick] (5) to (2);\draw[->,thick] (6) to (3);
\draw[->,thick] (4) to (2);
\draw[->,thick] (2) to (1);
\draw[->,thick] (6) to (4);
\draw[->,thick] (4) to (3);
\draw[->,thick] (3) to (1);
\draw[->,thick] (4) to (1);
\end{tikzpicture}}};
\node (x6) [align=center] at (0,6) {$x_6=\hat{1}$ \\ \scalebox{0.5}{\begin{tikzpicture}[scale=.8]
\node (1) [circle,thick,draw,fill=join,text=white] at (0,0) {1};
\node (2) [circle,thick,draw,fill=join,text=white]  at (-1,1) {2};
\node (3) [circle,thick,draw,fill=join,text=white] at (1,1) {3};
\node (4) [circle,thick,draw,fill=join,text=white] at (0,2) {4};
\node (5) [circle,thick,draw,fill=join,text=white] at (-1,3) {5};
\node (6) [circle,thick,draw,fill=join,text=white] at (1,3) {6};
\draw[->,thick] (5) to (4);\draw[->,thick] (5) to (2);\draw[->,thick] (6) to (3);
\draw[->,thick] (4) to (2);
\draw[->,thick] (2) to (1);
\draw[->,thick] (6) to (4);
\draw[->,thick] (4) to (3);
\draw[->,thick] (3) to (1);
\draw[->,thick] (4) to (1);
\end{tikzpicture}}};
\node (j3) [align=center]  at (1,2) {\scalebox{0.5}{\begin{tikzpicture}[scale=.8]
\node (1) [circle,thick,draw,fill=join,text=white] at (0,0) {1};
\node (2) [circle,thick,draw,fill=meet]  at (-1,1) {2};
\node (3) [circle,thick,draw,fill=join,text=white] at (1,1) {3};
\node (4) [circle,thick,draw,fill=meet] at (0,2) {4};
\node (5) [circle,thick,draw,fill=meet] at (-1,3) {5};
\node (6) [circle,thick,draw,fill=meet] at (1,3) {6};
\draw[->,thick] (5) to (4);\draw[->,thick] (5) to (2);\draw[->,thick] (6) to (3);
\draw[->,thick] (4) to (2);
\draw[->,thick] (2) to (1);
\draw[->,thick] (6) to (4);
\draw[->,thick] (4) to (3);
\draw[->,thick] (3) to (1);
\draw[->,thick] (4) to (1);
\end{tikzpicture}} \\ $\textcolor{join}{j_3}$};
\node (j5) [align=center] at (-2,1) {\scalebox{0.5}{\begin{tikzpicture}[scale=.8]
\node (1) [circle,thick,draw,fill=meet] at (0,0) {1};
\node (2) [circle,thick,draw,fill=white]  at (-1,1) {2};
\node (3) [circle,thick,draw,fill=meet] at (1,1) {3};
\node (4) [circle,thick,draw,fill=white] at (0,2) {4};
\node (5) [circle,thick,draw,fill=join,text=white] at (-1,3) {5};
\node (6) [circle,thick,draw,fill=meet] at (1,3) {6};
\draw[->,thick] (5) to (4);\draw[->,thick] (5) to (2);\draw[->,thick] (6) to (3);
\draw[->,thick] (4) to (2);
\draw[->,thick] (2) to (1);
\draw[->,thick] (6) to (4);
\draw[->,thick] (4) to (3);
\draw[->,thick] (3) to (1);
\draw[->,thick] (4) to (1);
\end{tikzpicture}} \\ $\textcolor{join}{j_5}$};
\node (j6) [align=center] at (2,1) {\scalebox{0.5}{\begin{tikzpicture}[scale=.8]
\node (1) [circle,thick,draw,fill=meet] at (0,0) {1};
\node (2) [circle,thick,draw,fill=meet]  at (-1,1) {2};
\node (3) [circle,thick,draw,fill=white] at (1,1) {3};
\node (4) [circle,thick,draw,fill=white] at (0,2) {4};
\node (5) [circle,thick,draw,fill=meet] at (-1,3) {5};
\node (6) [circle,thick,draw,fill=join,text=white] at (1,3) {6};
\draw[->,thick] (5) to (4);\draw[->,thick] (5) to (2);\draw[->,thick] (6) to (3);
\draw[->,thick] (4) to (2);
\draw[->,thick] (2) to (1);
\draw[->,thick] (6) to (4);
\draw[->,thick] (4) to (3);
\draw[->,thick] (3) to (1);
\draw[->,thick] (4) to (1);
\end{tikzpicture}} \\ $\textcolor{join}{j_6}$};
\node (m1)  [align=center]  at (0,5) {\scalebox{0.5}{\begin{tikzpicture}[scale=.8]
\node (1) [circle,thick,draw,fill=meet] at (0,0) {1};
\node (2) [circle,thick,draw,fill=white]  at (-1,1) {2};
\node (3) [circle,thick,draw,fill=white] at (1,1) {3};
\node (4) [circle,thick,draw,fill=white] at (0,2) {4};
\node (5) [circle,thick,draw,fill=join,text=white] at (-1,3) {5};
\node (6) [circle,thick,draw,fill=join,text=white] at (1,3) {6};
\draw[->,thick] (5) to (4);\draw[->,thick] (5) to (2);\draw[->,thick] (6) to (3);
\draw[->,thick] (4) to (2);
\draw[->,thick] (2) to (1);
\draw[->,thick] (6) to (4);
\draw[->,thick] (4) to (3);
\draw[->,thick] (3) to (1);
\draw[->,thick] (4) to (1);
\end{tikzpicture}} \\ $\textcolor{meet}{m_1}$};
\node (m2)  [align=center]  at (2,3) {$\textcolor{meet}{m_2}$ \\ \scalebox{0.5}{\begin{tikzpicture}[scale=.8]
\node (1) [circle,thick,draw,fill=join,text=white] at (0,0) {1};
\node (2) [circle,thick,draw,fill=meet]  at (-1,1) {2};
\node (3) [circle,thick,draw,fill=join,text=white] at (1,1) {3};
\node (4) [circle,thick,draw,fill=white] at (0,2) {4};
\node (5) [circle,thick,draw,fill=meet] at (-1,3) {5};
\node (6) [circle,thick,draw,fill=join,text=white] at (1,3) {6};
\draw[->,thick] (5) to (4);\draw[->,thick] (5) to (2);\draw[->,thick] (6) to (3);
\draw[->,thick] (4) to (2);
\draw[->,thick] (2) to (1);
\draw[->,thick] (6) to (4);
\draw[->,thick] (4) to (3);
\draw[->,thick] (3) to (1);
\draw[->,thick] (4) to (1);
\end{tikzpicture}}};
\node (m3)  [align=center]  at (-2,3) {$\textcolor{meet}{m_3}$ \\ \scalebox{0.5}{\begin{tikzpicture}[scale=.8]
\node (1) [circle,thick,draw,fill=join,text=white] at (0,0) {1};
\node (2) [circle,thick,draw,fill=join,text=white]  at (-1,1) {2};
\node (3) [circle,thick,draw,fill=meet] at (1,1) {3};
\node (4) [circle,thick,draw,fill=white] at (0,2) {4};
\node (5) [circle,thick,draw,fill=join,text=white] at (-1,3) {5};
\node (6) [circle,thick,draw,fill=meet] at (1,3) {6};
\draw[->,thick] (5) to (4);\draw[->,thick] (5) to (2);\draw[->,thick] (6) to (3);
\draw[->,thick] (4) to (2);
\draw[->,thick] (2) to (1);
\draw[->,thick] (6) to (4);
\draw[->,thick] (4) to (3);
\draw[->,thick] (3) to (1);
\draw[->,thick] (4) to (1);
\end{tikzpicture}}};
\node (m5)  [align=center] at (1,5) {$\textcolor{meet}{m_5}$ \\ \scalebox{0.5}{\begin{tikzpicture}[scale=.8]
\node (1) [circle,thick,draw,fill=join,text=white] at (0,0) {1};
\node (2) [circle,thick,draw,fill=join,text=white]  at (-1,1) {2};
\node (3) [circle,thick,draw,fill=join,text=white] at (1,1) {3};
\node (4) [circle,thick,draw,fill=join,text=white] at (0,2) {4};
\node (5) [circle,thick,draw,fill=meet] at (-1,3) {5};
\node (6) [circle,thick,draw,fill=join,text=white] at (1,3) {6};
\draw[->,thick] (5) to (4);\draw[->,thick] (5) to (2);\draw[->,thick] (6) to (3);
\draw[->,thick] (4) to (2);
\draw[->,thick] (2) to (1);
\draw[->,thick] (6) to (4);
\draw[->,thick] (4) to (3);
\draw[->,thick] (3) to (1);
\draw[->,thick] (4) to (1);
\end{tikzpicture}}};
\draw[-,thick,color=join,line width=2pt] (x0) to node[midway, left] {$1$} (x1) to node[midway, left] {$2$} (x2)  to node[midway, left] {$3$} (x3) to node[midway, left] {$4$} (x4) to node[midway, left] {$5$} (x5) to node[midway, left] {$6$} (x6);
\draw[-,thick,line width=2pt] (x1) to node[midway, right] {$3$} (j3) to node[midway, right] {$2$} (x3);
\draw[-,thick,line width=2pt] (x4) to node[midway, right] {$6$} (m5) to node[midway, right] {$5$} (x6);
\draw[-,thick] (j6) to node[midway, right] {$1$} (m2) to node[midway, right] {$2$} (m5);
\draw[-,thick] (j6) to node[midway, right] {$5$} (m1) to  node[midway, left] {$1$} (x6);
\draw[-,thick] (j5) to node[midway, left] {$1$} (m3) to node[midway, left] {$3$} (x5);
\draw[-,thick] (j5) to node[midway, left] {$6$} (m1);
\draw[-,thick] (x2) to node[midway, left] {$5$} (m3);
\draw[-,thick] (j3) to node[midway, right] {$6$} (m2);
\draw[-,thick] (x0) to node[midway, below] {$5$} (j5);
\draw[-,thick] (x0) to node[midway, below] {$6$} (j6);
\end{tikzpicture}
\caption{A (trim) Cambrian lattice for $A_3$ for bipartite Coxeter element.  The fixed maximal-length chain is colored blue and labeled by $\{x_i\}_{i=0}^6$.  The corresponding join-irreducible elements $\{j_i\}_{i=1}^6$ are in blue and the meet-irreducible elements $\{m_i\}_{i=1}^6$ are in orange.  Each element of the Cambrian lattice is also represented using~\Cref{thm:extreme_representation} as its corresponding maximal orthogonal pair $(X,Y)$, drawn directly on the Galois graph---the sets $X$ are colored blue, corresponding to the join-irreducible elements below the element, while the sets $Y$ are colored orange, corresponding to the meet-irreducible elements above the element.  Edges are labeled according to~\Cref{labelling} (in order, the labels $1,2,3,4,5,6$ correspond to the reflections $(23),(13),(24),(14),(34),(12)$).}
\label{fig:tamari_extremal}
\end{figure}

\subsection{Extremal lattices from representation theory}
\label{sec:rep_theory}

In this section, we briefly describe a source of extremal lattices in the representation theory of finite-dimensional algebras: namely, as the lattices of torsion classes for suitably chosen algebras.  The lattices that can be
obtained in this way include the Cambrian lattices of simply-laced type \cite{IT}, but also many more.  (We will show in~\Cref{cor:rep_are_trim} that these extremal lattices are actually trim.)

Let $A$ be a finite-dimensional algebra over a field $k$.  Write $\mod A$ for
the catgory of finite-dimensional left $A$-modules.  A module is called \defn{indecomposable} if it is not isomorphic to a direct sum of two non-zero
modules.  The algebra $A$ is said to be of \defn{finite representation type} if
it has only a finite number of isomorphism classes of indecomposable modules.  The category $\mod A$ is said to \defn{have no cycles} if there is no sequence of
indecomposable modules $M_0, M_1, \dots, M_r=M_0$ and nonisomorphisms $f_i:M_i \rightarrow M_{i+1}$ for $0 \leq i \leq r-1$.  

A \defn{full additive subcategory} of $\mod A$ is a subcategory of $\mod A$
which consists of direct sums of copies of some subset of the indecomposable
$A$-modules (and with morphisms inherited from the module category).
A \defn{torsion pair} in $\mod A$ is a pair of full, additive subcategories $(X,Y)$ such
that for $M\in X$ and $N\in Y$, we have $\Hom(M,N)=0$, and $X$ and $Y$ are both
maximal with respect to this condition.  The subcategory $X$ is called the
\defn{torsion class}, while the subcategory $Y$ is called the \defn{torsion-free class}.

Define a graph $G(A)$ whose vertices correspond to the indecomposable modules of $A$, and with an arrow from $M$ to $N$ if and only if $\Hom(M,N)\ne 0$.  The following theorem now follows from unwinding the definitions.

\begin{theorem} If $A$ is representation finite and $\mod A$ has no cycles, then $G(A)$ can be viewed as a Galois graph.  The elements of the extremal lattice associated to $G(A)$ are naturally the torsion pairs of $A$, ordered with respect to inclusion of torsion classes (or reverse inclusion of torsion-free classes).\label{thm:rep_are_extremal}\end{theorem}

This partial order (torsion pairs, ordered by inclusion of torsion classes; which can also be described as inclusion order on torsion classes) has been considered in \cite{IRTT,Ringel,GM2015,DIRRT}.  We refine~\Cref{thm:rep_are_extremal} in~\Cref{cor:rep_are_trim} to show that the lattices in question are in fact trim.

\begin{example}
We consider~\Cref{ex:cambrian} from the representation-theoretic perspective.  If $A$ is the path algebra of a Dynkin quiver, it is representation finite
and $\mod A$ has no cycles; the Cambrian lattices of types $A$, $D$, and $E$
can be obtained in this way \cite{IT}.
\end{example}

\subsection{Left modular lattices}
\label{sec:left_modular}

Left modular lattices generalize distributive lattices while still allowing us to find natural labellings of cover relations by join- and meet-irreducibles.

For any finite lattice $\LL$ and elements $y \leq z$, we always have the \defn{modular inequality}
\[(y\vee x)\wedge z \geq y\vee(x\wedge z).\]

\begin{definition}
An element $x$ of a lattice $\LL$ is called \defn{left modular} if for any $y\leq z$
we have the equality
\[(y\vee x)\wedge z=y\vee(x\wedge z).\]
A lattice is called \defn{left modular} if it has a maximal chain of left modular elements.
\label{def:left_modular}
\end{definition}

\begin{lemma}
\label{lem:left_mod_cover}
If $(y\vee x)\wedge z=y\vee(x\wedge z)$ for all covers $y\lessdot z$, then $x$ is left modular.
\end{lemma}
\begin{proof} Suppose $x$ is not left modular, as witnessed by the fact that $p = y \vee (x \wedge z) < (y \vee x) \wedge z = q$.  Choose $p \leq p' \lessdot q' \leq q$.  Then $p' \vee (x \wedge q') = p'$ and $(p' \vee x) \wedge q' = q'$.
\end{proof}

\subsubsection{Labelling of Cover Relations}
Let $\LL$ be left modular, with a given left modular chain
\[\hat 0 = x_0 \lessdot x_1 \lessdot \dots \lessdot x_n=\hat 1.\]
We assign to each join-irreducible $j \in \JJ$ and to each meet-irreducible $m \in \MM$ a label:
\begin{align*}\beta_\JJ(j)&:=\min \{ i \mid j\leq x_i\} \\ \beta_\MM(m)&:=\max \{ i \mid m \geq x_{i-1} \}.\end{align*}
Both $\beta_\JJ$ and $\beta_\MM$ are surjections to $[n]$.  This leads to three equivalent definitions of a labelling of the cover relations of $\LL$.

\begin{theorem}[{\cite{liu1999left}}]
  The following three labels associated to a cover relation $y\lessdot z$ in a left modular
  lattice $\LL$ are equal: \begin{enumerate}
  \item $\min\{ \beta_\JJ(j) \mid j\in \JJ,  y\vee j=z\}$,
  \item $\min \{ i \mid y \vee x_i \wedge z =z\}$,
  \item $\max\{ \beta_\MM(m) \mid m\in \MM, z\wedge m=y\}$.
  \end{enumerate}
We define $\gamma(y\lessdot z)$ to be this common value, and call $\gamma$ a \defn{left modular labelling}.
\label{labelling}
\end{theorem}


The equality of (1) and (2) was proven in~\cite{liu1999left}, and the equality
of (2) and (3) is dual; a proof is given in Lemma 2.1 of~\cite{mcnamara2006poset}.   This labelling recovers the labelling we have already defined for distributive lattices by associating poset elements with the order in which they are added in the given maximal chain.

\subsubsection{EL-Labellings}
An edge-labelling of the Hasse diagram of a poset $\Po$ is called an \defn{EL-labelling} \cite{Bj}
if, for any $x,y \in \Po$ with $x<y$, \begin{itemize}\item
there is a unique unrefinable chain from
$x$ to $y$ such that the word obtained by reading its labels from bottom
to top of the chain yields a word which is weakly increasing, and
\item this word lexicographically precedes the word corresponding to
  any other unrefinable chain from $x$ to $y$.  \end{itemize}

The existence of an EL-labelling of a poset implies that its order
complex is shellable (which gives good control over its homotopy type and M\"obius function, though this will not be especially relevant for us) \cite{Bj}. It was shown in~\cite{liu1999left} that if $\LL$ is left modular, the labelling
$\gamma$ is an EL-labelling.  In fact, more is true.

\subsubsection{Interpolating Labellings}
We say that an EL-labelling $\lambda$ is an \defn{interpolating labelling} if whenever we have $x,y,z$ in $\LL$ with $x\lessdot y\lessdot z$, either:
\begin{itemize}
\item  $\lambda(x\lessdot y)<\lambda(y\lessdot z)$ or
\item if the increasing chain from $x$ to $z$ is $x=x_0\lessdot x_1 \lessdot \dots\lessdot x_r=z$, then $\lambda(x_0\lessdot x_1)=\lambda(y\lessdot z)$ and $\lambda (x_{r-1}\lessdot x_r) = \lambda(x\lessdot y)$.
 \end{itemize}

\begin{example}
	We may describe Cambrian lattices using flips on $2$-colored factorizations of a Coxeter element~\cite{speyer2013acyclic,stump2015cataland}.  It is shown in~\cite{stump2015cataland} that the natural labelling on the edges of a Cambrian lattice coming from these factorizations is an EL-labelling; the flip definition allows us to easily see that this labelling also satisfies the interpolating condition.
\end{example}

\begin{theorem}{\cite{mcnamara2006poset}}
	If $\LL$ is a left modular lattice, then a left modular labelling $\gamma$ is an interpolating labelling.  Conversely, if a lattice $\LL$ admits an interpolating labelling, it is left modular.
	\label{thm:interpolating}
\end{theorem}

\section{Trim Lattices}
\label{sec:trim}
We have seen in~\Cref{sec:extremal,sec:left_modular} that both extremal and left modular lattices can be thought of as different generalizations of distributive lattices.  A lattice that is \emph{both} extremal \emph{and} left modular should therefore have even more in common with distributive lattices.  We note that not all extremal lattices are left modular (such as the lattice on the left of~\Cref{fig:extremal_no_global}), and that not all left modular lattices are extremal (such as the lattice on the right of~\Cref{fig:extremal_no_global}).

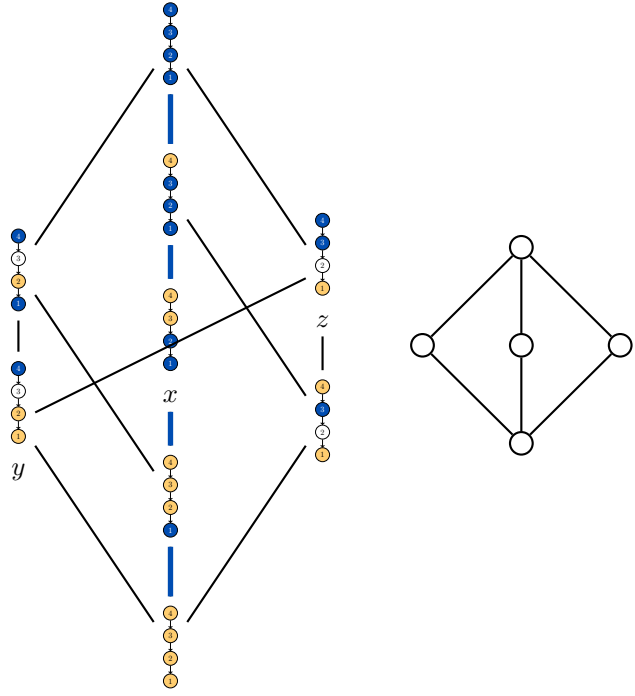
\begin{figure}[htbp]
	\raisebox{-0.5\height}{\begin{tikzpicture}[scale=2]
\node (a) [align=center] at (0,0) {\scalebox{0.3}{\begin{tikzpicture}
\node (1) [circle,thick,draw,fill=meet] at (0,0) {1};
\node (2) [circle,thick,draw,fill=meet] at (0,1) {2};
\node (3) [circle,thick,draw,fill=meet] at (0,2) {3};
\node (4) [circle,thick,draw,fill=meet] at (0,3) {4};
\draw[->,thick] (2) to (1);
\draw[->,thick] (3) to (2);
\draw[->,thick] (4) to (3);
\end{tikzpicture}}};
\node (b) [align=center] at (-1,1.5) {\scalebox{0.3}{\begin{tikzpicture}
\node (1) [circle,thick,draw,fill=meet] at (0,0) {1};
\node (2) [circle,thick,draw,fill=meet] at (0,1) {2};
\node (3) [circle,thick,draw,fill=white] at (0,2) {3};
\node (4) [circle,thick,draw,fill=join,text=white] at (0,3) {4};
\draw[->,thick] (2) to (1);
\draw[->,thick] (3) to (2);
\draw[->,thick] (4) to (3);
\end{tikzpicture}} \\ $y$};
\node (c) [align=center] at (0,1) {\scalebox{0.3}{\begin{tikzpicture}
\node (1) [circle,thick,draw,fill=join,text=white] at (0,0) {1};
\node (2) [circle,thick,draw,fill=meet] at (0,1) {2};
\node (3) [circle,thick,draw,fill=meet] at (0,2) {3};
\node (4) [circle,thick,draw,fill=meet] at (0,3) {4};
\draw[->,thick] (2) to (1);
\draw[->,thick] (3) to (2);
\draw[->,thick] (4) to (3);
\end{tikzpicture}}};
\node (d) [align=center] at (1,1.5) {\scalebox{0.3}{\begin{tikzpicture}
\node (1) [circle,thick,draw,fill=meet] at (0,0) {1};
\node (2) [circle,thick,draw,fill=white] at (0,1) {2};
\node (3) [circle,thick,draw,fill=join,text=white] at (0,2) {3};
\node (4) [circle,thick,draw,fill=meet] at (0,3) {4};
\draw[->,thick] (2) to (1);
\draw[->,thick] (3) to (2);
\draw[->,thick] (4) to (3);
\end{tikzpicture}}};
\node (e) [align=center] at (-1,2.5) {\scalebox{0.3}{\begin{tikzpicture}
\node (1) [circle,thick,draw,fill=join,text=white] at (0,0) {1};
\node (2) [circle,thick,draw,fill=meet] at (0,1) {2};
\node (3) [circle,thick,draw,fill=white] at (0,2) {3};
\node (4) [circle,thick,draw,fill=join,text=white] at (0,3) {4};
\draw[->,thick] (2) to (1);
\draw[->,thick] (3) to (2);
\draw[->,thick] (4) to (3);
\end{tikzpicture}}};
\node (f) [align=center] at (0,2) {\scalebox{0.3}{\begin{tikzpicture}
\node (1) [circle,thick,draw,fill=join,text=white] at (0,0) {1};
\node (2) [circle,thick,draw,fill=join,text=white] at (0,1) {2};
\node (3) [circle,thick,draw,fill=meet] at (0,2) {3};
\node (4) [circle,thick,draw,fill=meet] at (0,3) {4};
\draw[->,thick] (2) to (1);
\draw[->,thick] (3) to (2);
\draw[->,thick] (4) to (3);
\end{tikzpicture}} \\ $x$};
\node (g) [align=center] at (1,2.5) {\scalebox{0.3}{\begin{tikzpicture}
\node (1) [circle,thick,draw,fill=meet] at (0,0) {1};
\node (2) [circle,thick,draw,fill=white] at (0,1) {2};
\node (3) [circle,thick,draw,fill=join,text=white] at (0,2) {3};
\node (4) [circle,thick,draw,fill=join,text=white] at (0,3) {4};
\draw[->,thick] (2) to (1);
\draw[->,thick] (3) to (2);
\draw[->,thick] (4) to (3);
\end{tikzpicture}} \\ $z$};
\node (h) [align=center] at (0,3) {\scalebox{0.3}{\begin{tikzpicture}
\node (1) [circle,thick,draw,fill=join,text=white] at (0,0) {1};
\node (2) [circle,thick,draw,fill=join,text=white] at (0,1) {2};
\node (3) [circle,thick,draw,fill=join,text=white] at (0,2) {3};
\node (4) [circle,thick,draw,fill=meet] at (0,3) {4};
\draw[->,thick] (2) to (1);
\draw[->,thick] (3) to (2);
\draw[->,thick] (4) to (3);
\end{tikzpicture}}};
\node (i) [align=center] at (0,4) {\scalebox{0.3}{\begin{tikzpicture}
\node (1) [circle,thick,draw,fill=join,text=white] at (0,0) {1};
\node (2) [circle,thick,draw,fill=join,text=white] at (0,1) {2};
\node (3) [circle,thick,draw,fill=join,text=white] at (0,2) {3};
\node (4) [circle,thick,draw,fill=join,text=white] at (0,3) {4};
\draw[->,thick] (2) to (1);
\draw[->,thick] (3) to (2);
\draw[->,thick] (4) to (3);
\end{tikzpicture}}};
\draw[-,thick] (a) to  (b) to  (e) to  (i);
\draw[-,thick] (a) to  (c) to  (e);
\draw[-,thick] (b) to (g) to (i);
\draw[-,thick] (a) to (d) to (h);
\draw[-,thick] (c) to (f) to (h) to (i);
\draw[-,thick] (d) to  (g);
\draw[-,thick,color=join,line width=2pt] (a) to (c) to  (f) to (h) to (i);
\end{tikzpicture}}
\hspace{2em}
\raisebox{-0.5\height}{\begin{tikzpicture}[scale=1.3]
\node (a) [circle,thick,minimum size=3pt,inner sep=3pt,draw,fill=white] at (0,0) {};
\node (b) [circle,thick,minimum size=3pt,inner sep=3pt,draw,fill=white]  at (-1,1) {};
\node (c) [circle,thick,minimum size=3pt,inner sep=3pt,draw,fill=white] at (0,1) {};
\node (d) [circle,thick,minimum size=3pt,inner sep=3pt,draw,fill=white]  at (1,1) {};
\node (e) [circle,thick,minimum size=3pt,inner sep=3pt,draw,fill=white]  at (0,2) {};
\draw[-,thick] (a) to (b) to (e);
\draw[-,thick] (a) to (c) to (e);
\draw[-,thick] (a) to (d) to (e);
\end{tikzpicture}}
\caption{On the left is an extremal lattice $\LL$ that is not left modular---the failure of the left-modularity of the element $x$ is witnessed by the elements marked $y$ and $z$, since $(y\vee x)\wedge z > y\vee(x\wedge z)$.  We can also see that $\LL$ is not trim using~\Cref{thm:overlapping}: the cover relation $y < z$ is non-overlapping, since $y_\MM \cap z_\JJ  =\{1,2\} \cap \{3,4\} = \emptyset$.  As $\LL$ is not trim, we observe that \Cref{cor:independent} fails---it has $9$ elements, but its Galois graph only has $8$ independent sets.  On the right is a left modular lattice that is not extremal.}
\label{fig:extremal_no_global}
\end{figure}

\begin{definition}[\cite{thomas2006analogue}]
A lattice $\LL$ is called \defn{trim} if it is both extremal and left modular.
\end{definition}

Examples of trim lattices include distributive lattices, finite Cambrian lattices \cite{IT}, the $m$-Cambrian lattices of~\cite{stump2015cataland}, as well as all finite intervals in arbitrary Cambrian lattices~\cite{muhle2016trimness} and the $m$-Tamari lattices of Bergeron and Pr{\'e}ville-Ratelle~\cite{bergeron2011higher,bergeron2012combinatorics}.

Let $\LL$ be a trim lattice.  We fix once and for all a choice of
left modular chain
\[\hat 0=\x_0 \lessdot \x_1 \lessdot \cdots \lessdot \x_n=\hat 1.\]
Since $\LL$ is extremal and left modular, its elements inherit a representation as maximal orthogonal pairs from~\Cref{sec:extremal} and $\LL$ inherits the left modular labelling $\gamma$ from~\Cref{labelling}.  Note that $\beta_\JJ$ and $\beta_\MM$ are now \emph{bijections} to $[n]$, which is partially ordered as the Galois poset $P(\LL)$.

\subsection{Overlapping maximal orthogonal pairs}

We give a useful characterization of trim lattices among extremal lattices using maximal orthogonal pairs.
\begin{definition}
\label{def:overlapping}
A relation $y<z$ in an extremal lattice $\LL$ is \defn{overlapping} if \[y_\MM\cap z_\JJ \neq \emptyset.\]
\end{definition}

\begin{example}
For the extremal (but not left modular) lattice on the left of~\Cref{fig:extremal_no_global}, observe that $y < z$ is non-overlapping, since $y_\MM \cap z_\JJ  =\{1,2\} \cap \{3,4\} = \emptyset$.
\label{ex:nonoverlapping}
\end{example}
As the following theorem shows,~\Cref{ex:nonoverlapping} is representative of the only way extremal lattices can fail to be trim.

\begin{theorem}
An extremal lattice $\LL$ is trim if and only if every relation is overlapping if and only if every cover relation is overlapping.
\label{thm:overlapping}
\end{theorem}

\begin{proof}
We first show that all relations $y<z$ are overlapping in a trim lattice $\LL$.  Representing elements of $\LL$ by their maximal orthogonal pairs, we may reduce to the case that $(X,Y) \lessdot (X',Y')$.  Using~\Cref{labelling}, the label of this cover is an element of $X'\setminus X$ by the labelling $\beta_\JJ$, and an element of $Y \setminus Y'$ by the labelling $\beta_\MM$.  The two sets must therefore intersect.

We now show that if every pair $y<z$ is overlapping, then $\LL$ is left modular.  Note that since $\LL$ is extremal, we have already fixed a maximal-length chain $\hat{0}=x_0\lessdot x_1 \lessdot \cdots \lessdot x_n = \hat{1}$.  We will show that this chain consists of left modular elements $x_k$; by~\Cref{lem:left_mod_cover}, it suffices to check the left modularity condition on cover relations.

We first assign a label to a cover relation.   For a cover $y \lessdot z$, we claim that $y_\MM \cap z_\JJ  = \{i\}$ and $\min(z_\JJ \setminus y_\JJ)=\max(y_\MM\setminus z_\MM)=i$.  For otherwise, let $p=\min(z_\JJ\setminus y_\JJ)$ and $q=\max(y_\MM \setminus z_\MM)$ and set $y'_\JJ=y_\JJ \cup \{p\}$ $z'_\MM=z_\MM \cup \{q\}$.  By assumption $p<q$, and so $(y'_\JJ,z'_\MM)$ is an orthogonal pair which is contained in some maximal orthogonal pair $(y''_\JJ,z''_\MM)$.  But this maximal orthogonal pair $(y''_\JJ,z''_\MM)$ is now properly between $y$ and $z$, contradicting that $y\lessdot z$.  Therefore, $z_\JJ \cap y_\MM = \{i\}$ (and $i$ turns out to be the correct left modular labelling of the cover relation).

To show $x_k$ is left modular, we now show that \begin{align*} y \vee (x_k \wedge z) &= (y \vee x_k) \wedge z = y \text{ if } k<i \text{ and}  \\ y \vee (x_k \wedge z) &= (y \vee x_k) \wedge z = z \text{ if } k \geq i.\end{align*}  It suffices to check that $y \vee (x_i \wedge z) = z$ and $(y \vee x_{i-1}) \wedge z = y$.  For the first, it is clear that $(x_i \wedge z)_\JJ$ contains $i$, and the desired claim follows.  The other follows by duality.
\end{proof}

Overlapping cover relations allow us to recover the edge labellings of a trim lattice from the representation of elements as maximal orthogonal pairs.

\begin{corollary}
\label{cor:contains}
For $\LL$ a trim lattice, the left modular label assigned to a cover relation $\gamma(y \lessdot z)$ is the unique element in the intersection $y_\MM \cap z_\JJ$.
\end{corollary}

\subsection{Structural properties of trim lattices}
\label{sec:structure}
We recall some structural properties of trim lattices.  Of particular importance is that intervals in trim lattices inherit trimness (in contrast, intervals of extremal lattices are not usually again extremal lattices---indeed, Markowsky shows that any lattice can be embedded as an interval in an extremal lattice \cite[Theorem 14]{markowsky1992primes} ).

\begin{theorem}[{\cite[Theorem 1]{thomas2006analogue}}]\label{thm:intervals}
If $\LL$ is trim, then so is any interval of $\LL$.
\end{theorem}

Recall that the spine of a extremal lattice $\LL$ is the distributive lattice consisting of all elements lying on maximal-length chains of $\LL$.  For trim lattices, left modularity provides an alternative description of the spine.

\begin{theorem}\label{thm:spine}
An element of a trim lattice $\LL$ is left modular if and only if it lies on the spine.
\end{theorem}

\begin{proof}
That the elements of the spine of a trim lattice $\LL$ are all left modular follows from the proof of~\Cref{thm:overlapping} (and appears in~\cite[Lemma 5]{thomas2006analogue}).  We now show that all elements not on the spine are not left modular.  Suppose $x$ is not on the spine, so that there is some $i$ missing from $x_\JJ \cup x_\MM$.  We will show that $x$ does not satisfy left modularity with respect to the pair $(x_{i-1},x_i)$.  First, we check that $x_i \wedge (x \vee x_{i-1}) = x_i$.  For $k\leq i$, $k \not \in (x \vee x_{i-1})_\MM$ and so $\{1,2,\ldots,i\} \subseteq (x \vee x_{i-1})_\JJ$.  In particular, $x \vee x_{i-1} \geq x_i$, and so meeting with $x_i$ gives $x_i$.  The dual argument now gives that $(x_i \wedge x) \vee x_{i-1} = x_{i-1}$, which implies that $x$ is not left modular.
\end{proof}

By~\Cref{cor:galois_unique}, any choice of maximal-length chain in an extremal lattice gives an isomorphic Galois graph.  The following proposition supplements this result for trim lattices by showing that any choice of maximal-length chain associates the same join-irreducible element (and hence the same meet-irreducible element) to a cover relation of $\LL$.

\begin{proposition}
Let $(x_i)_{i=0}^n$ and $(x'_i)_{i=0}^n$ be two chains of maximal length in a trim lattice $\LL$, inducing maps $\beta_\JJ$, $\beta'_\JJ$ to $[n]$ and  left modular labellings $\gamma$ and $\gamma'$.  Suppose $\gamma(y \lessdot z)=\beta_\JJ(a)$, and $\gamma'(y \lessdot z)=\beta'_\JJ(b)$ with $a,b \in \JJ$.  Then $a=b$.
\label{thm:trim_labels}
\end{proposition}

\begin{proof}
As in the proof of~\Cref{cor:galois_unique}, we may assume that the two chains differ in only one position (say, the $j$th), since any two chains in a distributive lattice differ by a sequence of such differences.  Let $y \vee x_i \wedge z = z$ and $y \vee x'_k \wedge z=z$ with $i$ and $k$ minimal in the chains.  If $x_i$ and $x'_k$ both occur in both chains, then $a=b$, and we are done.  So we may assume that one of the indices, say $i$, is equal to the position $j$ where the two chains differ.

Observe that $k$ cannot also be equal to $j$; otherwise
$y \vee x_{j-1} \wedge z = y \vee (x_j \wedge x'_j) \wedge z = z$, contradicting minimality of $j$.  In fact, we see that $k=j+1$ is the only possibility.  The ordered sequence of join-irreducibles associated to the two chains can only be different in the $j$th and $(j{+}1)$st spots, and they must differ, so they differ by exchanging the two join-irreducibles in those spots.  Thus the join-irreducibles corresponding to the $j$th step in chain $(x_i)_{i=0}^n$ and the $(j{+}1)$st step in the $(x'_i)_{i=0}^n$ chain are the same, as desired.
\end{proof}

Recall that a \defn{modular lattice} is a (finite) lattice that satisfies
\[y \wedge z \lessdot z \text{ and } y \wedge z \lessdot y \text{ if and only if } z \lessdot (y \vee z) \text{ and } y \lessdot (y \vee z).\]
Upper and lower semimodular lattices are defined by preserving only one direction of the implication.  Trim lattices satisfy a kind of demi-semimodularity.

  \begin{proposition}[{\cite[Theorem 6]{thomas2006analogue}}]\label{demi-semi} Let $\LL$
    be a trim lattice.
    Suppose that
  $x$ covers $y$ and $z$, and that $\gamma(y\lessdot x)>\gamma(z\lessdot x)$.
    Then $y \wedge z \lessdot z$; and also the dual statement when $y$ and
    $z$ cover $x$.
\end{proposition}

We also have the following result, which says that, as for extremal lattices, the property of being a trim lattice is preserved under lattice quotients as defined in~\Cref{ssec:lattices}.
\begin{lemma}
Trim lattices are preserved under lattice quotients.
\end{lemma}

\begin{proof}
The property of being a left modular lattice is clearly preserved under lattice quotients, while
the corresponding statement for extremality is \Cref{lem:ext-quot}.
\end{proof}


\begin{remark}
There are three different ``quotients'' that can be naturally defined in this setting: a representation-theoretic, a lattice-theoretic, and a graph-theoretic quotient.  Given the lattice of torsion pairs of an algebra $A$ satisfying the hypotheses of \Cref{cor:rep_are_trim} (which is therefore, as we shall show, trim), we define a \defn{representation-theoretic quotient} by taking a quotient of $A$ by a two sided ideal, and considering the torsion pairs of the quotient algebra.  We can also take a lattice-theoretic quotient of $\LL$.  Finally, we define a \defn{graph-theoretic quotient} on the level of the Galois graph $G$ by deleting some subset of the vertices of $G$.  Then a representation-theoretic quotient induces a lattice-theoretic quotient \cite{DIRRT}, and a lattice-theoretic quotient induces a graph-theoretic quotient,
but there are lattice-theoretic quotients of trim lattices of torsion pairs which do not arise from representation-theoretic quotients, and there are graph-theoretic quotients which do not correspond to lattice theoretic quotients.
\end{remark}

\subsection{The trim recurrence}
\label{sec:decomposition}
We introduce a recursive decomposition of trim lattices that will underlie many of our subsequent proofs.  This decomposition is related to the recursive decompositions of~\cite[Theorem 15]{markowsky1992primes} and naturally generalizes the \defn{Cambrian recurrence} on Cambrian lattices.

\begin{lemma}
\label{lem:decomposition}
Let $\LL$ be a trim lattice, and let $\LL_1 := [\hat{0},m_1]$ and $\LL^1:=[j_1,\hat{1}]$.  Then \[\LL = \LL_1 \sqcup \LL^1.\]
\end{lemma}

\begin{proof}
By construction, the vertex $1$ in the Galois graph $G(\LL)$ is a sink.  For $\x \in \LL$ to be in neither $\LL_1$ nor $\LL^1$, $1$ cannot appear in either of $\x_\JJ$ or $\x_\MM$ (the maximal orthogonal pair corresponding to $\x$).  But $1$ can certainly be added to $\x_\JJ$, contradicting the assumption that $\x_\JJ$ was maximal.
\end{proof}

\begin{proposition}
Let $\LL$ be a trim lattice with Galois graph $G(\LL)$.  Then
\begin{itemize}
\item the Galois graph $G(\LL^1)$ is obtained by deleting the vertex $1$ from $G(\LL)$ (along with all edges to $1$), and
\item the Galois graph $G(\LL_1)$ is obtained by deleting all vertices and edges adjacent to $1$ in $G(\LL)$ (along with $1$ itself).
\end{itemize}
\label{lem:recursion}
\end{proposition}

\begin{proof}
The statement for $\LL^1$ follows from the definition of maximal orthogonal pair.  We now consider $\LL_1$.  Its set of join-irreducibles are the join-irreducibles of $\LL$ which lie below $m_1$; these correspond by definition to the vertices of $G(\LL)$ not adjacent to 1.   By \Cref{thm:intervals}, $\LL_1$ is trim, so its join-irreducibles and its meet-irreducibles are in natural bijection.  For those $i$ with $j_i\leq m_1$, write $m_i'$ for the corresponding meet-irreducible of $\LL_1$.

We show that $m_i'=m_i \wedge m_1$, following the proof of~\cite[Theorem 1]{thomas2006analogue}.  Let $z=m_i \wedge m_1$.  Since $z\leq m_i$ and $j_i \not \leq m_i$, $j_i \not \leq z$.  So $j_i \vee z \neq z$.  Let the increasing chain from $z$ to $j_i \vee z$ be $z= t_0 \lessdot t_1 \lessdot \cdots \lessdot t_r=j_i \vee z$.  We have the upper bound $t_r \leq m_1$, so that $t_1 \not \leq m_i$---since otherwise we have the contradiction $t_1 \leq m_1 \wedge m_i=t_0$.  By the labelling of $\LL_1$ by meet-irreducibles, $\gamma_{\LL_1}(t_0,t_1)\geq i$; similarly, by its labelling by join-irreducibles, $\gamma_{\LL_1}(t_{r-1},t_r) \leq i$.  Since the labels on the chain increase, the chain is a single covering relation, labelled in $\LL_1$ by $i$.

The element $z$ therefore lies below $m_i'$.  But any element at the bottom of an edge labeled $i$ in $\LL_1$ lies below both $m_1$ and $z$.  Therefore $z$ must be the meet-irreducible labelled $i$ in $\LL_1$, and so $z=m_i \wedge m_1=m_i'$.

The desired description of the Galois graph of $\LL_1$ follows, since if $j_i$ is a join-irreducible below $m_1$, then $j_i$ is below $m_k$ iff $j_i$ is below $m_k'=m_k\wedge m_1$.
\end{proof}

 As in the introduction, for each element $\y \in \LL$, we define its set of \defn{downward labels} \[\down(\y):=\left\{\gamma(\x \lessdot \y) : \text{ for all } \x \text{ such that } \x \lessdot \y\right\},\]
and its set of \defn{upward labels}
\[\up(\y) := \left\{\gamma(\y \lessdot z) : \text{ for all } z \text{ such that } \y \lessdot z\right\}.\]

For $\LL'$ an interval (and hence trim sublattice) of $\LL$ and an element $\x \in \LL'$, we write $\down_{\LL'}(\x)$ and $\up_{\LL'}(\x)$ for the restriction of the set of labels of $\x$ to the cover relations in $\LL'$.  Any label $i$ can belong to at most one of $\down(\x)$ and $\up(\x)$.

\begin{lemma}
\label{lem:trichotomy}
An element $\x \in \LL_1$ if and only if $1 \in \up(\x)$.
\end{lemma}

\begin{proof}
If $1 \in \up(\x)$, then $x \not \geq j_1$ and so $x \in \LL_1$.  If $x\in \LL_1$, then there is an increasing chain from $\x$ to $\x\vee j_1$, and---by definition of the labelling $\gamma$---the last cover relation in the chain would have the label $1$.  But then this chain would have to be of length one, so that $1\in \up(\x)$.
\end{proof}

\begin{lemma} If $y \in \LL_1$, then $y \vee j_1\gtrdot y$ and $\down(y \vee j_1)=\down(y)\cup \{1\}$.
\label{lem:down}
\end{lemma}

\begin{proof}
	By~\Cref{lem:trichotomy}, $1 \in \up(y)$.  Write $z=y \vee j_1$ for the element covering $y$ with $\gamma(y \lessdot z) = 1$.  Observe that $1 \in \down(z)$.

	Fix $i\in \down(y)$.  Let $x$ be the element covered by $y$ such that $\gamma(x \lessdot y) = i$.  Now, $x\lessdot y \lessdot z$, and since $\gamma$ is an interpolating labelling by~\Cref{thm:interpolating}, the edge down
    from $z$ on the increasing chain from $x$ to $z$ is labelled by
    $\gamma(x\lessdot y)=i$. (Note that $x\lessdot y \lessdot z$ is not the increasing chain from $x$ to $z$.)  So $\down(y) \subset \down(z)$.

    Finally, suppose $i\in \down(z) \setminus \{1\}$.  Let
    $\gamma(x\lessdot z)=i$.  By~\Cref{demi-semi}, $y\wedge x\lessdot y$, and by~\Cref{thm:interpolating} we have $\lambda(y\wedge x \lessdot x)=i$ so that $\down(z)=\down(y) \cup \{1\}$.
\end{proof}

\begin{remark}
The specialization of~\Cref{lem:down} to Cambrian lattices appears as~\cite[Lemma 2.8]{reading2007sortable}.
\end{remark}

\begin{lemma} If $x\in \LL^1$, then $\down_{\LL^1}(x)=\down_{\LL}(x)\setminus \{1\}$.
\label{lem:go_up_1}
\end{lemma}

\begin{proof} Edges from an element of $\LL^1$ down to an element not in
$\LL^1$ are necessarily labelled $1$, and it is exactly such edges whose labels must be removed from $\down_{\LL}(x)$ to obtain $\down_{\LL^1}(x)$.  (The number of such edges going down from a particular $x$ in $\LL'$ is either zero or one.  When we write $\down_{\LL}(x)\setminus \{1\}$, we do not mean to imply that $1$ is necessarily present in $\down_{\LL}(x)$.)
\end{proof}

\section{Rowmotion for Trim Lattices}
\label{sec:global_row}

In this section we prove~\Cref{thm:main_thm}.

\subsection{Descriptive labellings of trim lattices}
We first show that a left modular labelling of a trim lattice is descriptive, which allows us to define rowmotion for trim lattices.

\begin{proposition}\label{pone}
An element $\x$ of a trim lattice $\LL$ is determined by $\down(x)$, and by $\up(x)$: \[x=\bigvee_{i \in \down(x)} j_i = \bigwedge_{i \in \up(x)} m_i.\]
Furthermore, $(\x_\JJ,\x_\MM)$ is the unique maximal orthogonal pair such that $\down(x) \subseteq x_\JJ$ and $\up(x) \subseteq x_\MM$.
\end{proposition}

  \begin{proof}
    For $i\in \down(\x)$, we know that there is an edge $\y \lessdot \x$ such that $\x=\y\vee j_i$, which implies that $j_i\leq x$ and $j_i \not\leq y$.
    So $\x \geq \bigvee_{i\in \down(\x)} j_i$, and for any $\y \lessdot \x$, $\bigvee_{i\in \down(x)} j_i \not \leq \y$.  Thus $\x=\bigvee_{i \in \down(x)} j_i$.
    By \Cref{labelling}, the dual analysis gives the corresponding result for $\up(\x)$.

    As $x$ is determined by both $\down(x)$ and $\up(x)$, apply~\Cref{cor:contains} to conclude that $\x_\JJ$ must contain $\down(\x)$ and $\x_\MM$ must contain $\up(\x)$.
    \end{proof}

\begin{remark}
Note that a distributive lattice $J(\mathcal{Q})$ is descriptively labelled by $\mathcal{Q}$---indeed, $\down(\x)$ is the antichain in $\Po$ consisting of the maximal elements of $\x$, while $\up(\x)$ is the antichain consisting of the minimal elements in $\Po \setminus x$, and any antichain in $\Po$ can play either of these r\^oles.

In the special case when $\LL$ is a Cambrian lattice,~\Cref{pone} appears as~\cite[Lemma 2.5]{reading2007sortable}, and is a consequence of the bijection between sortable elements and noncrossing partitions.

    Thus, as in the introduction, the sets $\down(\x)$ and $\up(\x)$ simultaneously generalize the notion of ``antichain'' from distributive lattices and ``noncrossing partition'' from Cambrian lattices.
\end{remark}

We will show that the set of all downward labels is the same as the set of all upward labels by induction on the length of $\LL$, using the decomposition of $\LL$ stated in~\Cref{lem:decomposition}.

\begin{proposition}\label{ptwo}  Let $\LL$ be a trim lattice.  Then \[\Big\{ \down(\x) : \x \in \LL\Big\} = \Big\{ \up(\y) : \y \in \LL\Big\}.\]
\end{proposition}

\begin{proof}
	For any $\x \in \LL$, we find an element $\y$ such that $\down(\x)=\up(\y)$ by induction on the length $n$ of $\LL$.  The proof follows three disjoint cases refining~\Cref{lem:decomposition}.

	\medskip
	\emph{Case I: $x \in \LL_1$.}  Writing $\x^1 = \x \vee j_1$, we have $\down(\x^1)=\down(\x)\cup \{1\}$ by \Cref{lem:down}.  The interval $\LL^1$ is a trim lattice of length $n-1$ by~\Cref{thm:intervals}.  Since $\x^1 \in \LL^1$, by induction on length, we obtain an element $\y \in \LL^1$ such that $\up_{\LL^1}(\y) = \down_{\LL^1}(\x^1)$.  But $\up_{\LL^1}(\y)=\up(\y)$ and $\down_{\LL^1}(\x^1) = \down_{\LL}(\x^1) \setminus \{1\} = \down(\x)$.

 \medskip
    \emph{Case II: $x \in \LL^1$ and $1 \in \down(\x)$.}  Since $1 \in \down(\x)$, there is an element $x_1 \lessdot x$ such that $x_1 \vee j_1 = \x$.  By \Cref{lem:down}, $\down_{\LL_1}(\x_1)= \down_{\LL}(x) \setminus \{1\}$.  The interval $\LL_1$ is a trim lattice of strictly shorter length than $\LL$.  Since $x_1 \in \LL_1$, by induction on length, we obtain an element $\y \in \LL_1$ such that $\up_{\LL_1}(\y) = \down_{\LL_1}(\x_1)$.  But now viewing $\y$ as an element of $\LL$, we see that also $1 \in \up_{\LL}(\y)$.

	\medskip
	\emph{Case III: $\x \in \LL^1$ and $1 \not \in \down(\x)$.}  Since $\LL^1$ is a trim lattice of length $n-1$, by induction on length, we obtain an element $\y \in \LL^1$ such that $\up_{\LL^1}(\y) = \down_{\LL^1}(\x)=\down(\x)$.
\end{proof}

\Cref{eq:global_row} therefore defines a \defn{rowmotion} $\row^\gamma:\LL\to\LL$ on a trim lattice $\LL$ by \[\row^\gamma(\x):=\text{ the unique element } y \in \LL \text{ such that } \down(\x)=\up(\y).\]

\begin{theorem}\label{global_independent}
The definition of $\row^\gamma$ on a trim lattice $\LL$ does not depend on the choice of left modular labelling $\gamma$.
\end{theorem}

\begin{proof}
	By~\Cref{thm:trim_labels}, if we replace the label of a cover relation by the corresponding join-irreducible, then the labels do not depend on the initial choice of maximal chain.  The result follows.
\end{proof}

\subsection{Slow motion}

As in the introduction, a \defn{flip} at the label $i \in [n]$ is the permutation $\tog_i: \LL \to \LL$ defined by \[\tog_i (\y) := \begin{cases} \x & \text{if } i \in \down(\y) \text{ and } \gamma(\x \lessdot \y) = i,\\ z & \text{if } i \in \up(\y) \text{ and } \gamma(\y \lessdot z) = i, \\ \y & \text{ otherwise.} \end{cases}\]  In words, starting at the element $\y$ in the Hasse diagram of $\LL$, flipping at $i$ can be visualized as a walk along the unique edge incident to $\y$ labeled by $i$, should such an edge exist. 

{
\renewcommand{\thetheorem}{\ref{thm:main_thm}}
\begin{theorem}     Any trim lattice has a descriptive labelling, and rowmotion can be computed in slow motion.
\end{theorem}
\addtocounter{theorem}{-1}
}

  \begin{proof}
\Cref{pone,ptwo} prove that the left modular labelling of a trim lattice is descriptive.

We will check that flipping in the order of any linear extension of $P(\LL)$ defines the same permutation as rowmotion---since the maximal length chains in a trim lattice correspond to the linear extensions of the Galois poset $P(\LL)$, this will prove that rowmotion can be computed in slow motion.  Let $\mathbf{l}$ be a linear extension of $P(\LL)$.  Since the definition of $\row^\gamma$ does not depend on the left modular labelling $\gamma$ by~\Cref{global_independent}, we may assume the first element of $\mathbf{l}$ is $1$.

The proof is by induction on the length of $\LL$.  We again follow three disjoint cases refining~\Cref{lem:decomposition}.
\medskip

    \emph{Case I: $x \in \LL_1$.} At the first step when calculating $\prod_{i \in \la} \tog_{i} $, we walk to $y\gtrdot x$, where $y=x\vee j_1$.  So $\down(y)=\down(x)\cup \{1\}$ by~\Cref{lem:down} and $y\in \LL^1$ by~\Cref{lem:trichotomy}.  Applying the rest of the flips to $y$ has the effect of applying  $\prod_{\substack{i \in \la \\ i\neq 1}} \tog_{i} $ in $\LL^1$.  (Note that the only edges in $\LL$ leaving $\LL^1$ are labelled by 1, so they will never be taken.)  By induction, we obtain an element $z \in \LL^1$ such that $\up_{\LL^1}(z)=\down_{\LL^1}(y)$.  But $\up_{\LL^1}(z)=\up_\LL(z)$, and $\down_{\LL^1}(y)=\down_{\LL}(y)\setminus \{1\}=\down_\LL(x)$.

\medskip
    \emph{Case II:  $x \in \LL^1$ and $1 \in \down(\x)$.}   At the first step when calculating $\prod_{i \in \la} \tog_{i}(x)$, we walk to $y\lessdot x$, where $x=y\vee j_1$.  Since $1 \in \up(\y)$, $y\in \LL_1$ by~\Cref{lem:trichotomy}.  Then $\down(y) \cup \{1\}=\down(x)$ by~\Cref{lem:down}.  Applying the rest of the flips to $y$ has the effect of applying  $\prod_{\substack{i \in \la \\ i\neq 1}} \tog_{i} $ in $\LL_1$.  (Note that the only edges in $\LL$ leaving $\LL_1$ are labelled by 1, so they will never be taken.)  By induction, we obtain an element $z \in \LL_1$ such that $\up_{\LL_1}(z)=\down_{\LL_1}(y)$.  But $\up_{\LL_1}(z)\cup\{1\}=\up_\LL(z)$, and $\down_{\LL_1}(y)\cup\{1\}=\down_{\LL}(y)\cup \{1\}=\down_\LL(x)$.

\medskip
    \emph{Case III: $\x \in \LL^1$ and $1 \not \in \down(\x)$.}
By assumption, flipping at 1 has no effect.  As in Case I, the rest of the flips have the effect of applying $\prod_{\substack{i \in \la \\ i\neq 1}} \tog_{i}$ in $\LL^1$.  By induction, we obtain an element $z \in \LL^1$ such that $\up_{\LL^1}(z)=\down_{\LL^1}(x)$.  But $\down_{\LL^1}(x)=\down_\LL(x)$, and $\up_{\LL^1}(z)=\up_\LL(z)$.
\end{proof}

\section{The Independence Complex of a Trim Lattice}
\label{sec:labelcomplex}

In this section, motivated by Barnard's corresponding result for semidistributive lattices, we prove that the set of downward (or upward) labels in a trim lattice is a flag simplicial complex, that its 1-skeleton is the complement of the Galois graph, and that it contains exactly the independent sets of the Galois graph.

\begin{proposition}
For $\LL$ a trim lattice with left modular labelling $\gamma$, \[\gamma(\LL):=\Big\{\down(x) : x \in \LL\Big\}=\Big\{\up(x) : x \in \LL\Big\}\] is a simplicial complex.  We call $\gamma(\LL)$ the \defn{independence complex} of $\LL$.
\label{prop:ind_complex}
\end{proposition}

\begin{proof}
The equality of downward and upward labels in the definition of $\gamma(\LL)$ follows from~\Cref{ptwo}.  Let $A \in \gamma(\LL)$, so that $A=\down(x)$ for some $x \in \LL$.  For any subset $B\subseteq A$, we must find an element $y$ for which $\down(y)=B$.

 \medskip
	\emph{Case I: $x \in \LL_1$.}  We are done by induction, since $\LL_1$ is a trim lattice of shorter length than $\LL$.

 \medskip
    \emph{Case II: $x \in \LL^1$ and $1 \in \down(\x)$.}  Considering the element $x_1$ such that $x=x_1\vee j_1$, we see that $A \setminus \{1\} \in \gamma(\LL_1)$.  By induction, we may find $y_1 \in \LL_1$ with $\down(y_1)=B\setminus\{1\}$.  If $1 \in B$, take $y=j_1 \vee y_1$; otherwise, take $y=y_1$.

	\medskip
	\emph{Case III: $\x \in \LL^1$ and $1 \not \in \down(\x)$.}  Since $\LL^1$ has shorter length than $\LL$, we may find $y^1 \in \LL^1$ with $\down_{\LL^1}(y^1)=B$.  If $\down(y^1)=B$, take $y=y^1$; otherwise, $\down_{\LL^1}(y^1)=B \cup \{1\}$ and we may take $y$ to be the element in $\LL_1$ such that $y^1=j_1 \vee y$.
\end{proof}

\begin{example}
The independence graph and Galois graph for the Cambrian lattice of~\Cref{fig:tamari_extremal} are drawn in~\Cref{fig:cjcgalois}.
\end{example}

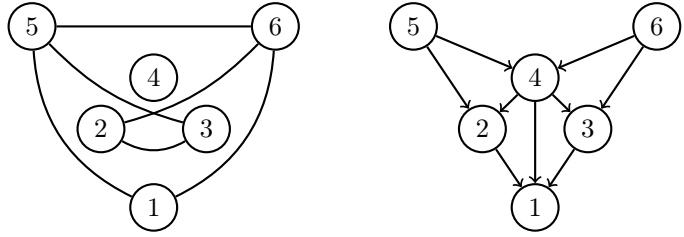
\begin{figure}[htbp]
\begin{center}
\begin{tikzpicture}[scale=.8]
\node (1) [circle,thick,draw] at (0,0) {1};
\node (3) [circle,thick,draw]  at (.867,1.3) {3};
\node (2) [circle,thick,draw] at (-.867,1.3) {2};
\node (4) [circle,thick,draw] at (0,2.15) {4};
\node (5) [circle,thick,draw] at (-2,3) {5};
\node (6) [circle,thick,draw] at (2,3) {6};
\draw[-,thick] (5) to (6);
\draw[-,thick] (2) to [bend right] (3);
\draw[-,thick] (1) to [bend right] (6);
\draw[-,thick] (1) to [bend left] (5);
\draw[-,thick] (2) to [bend right=15] (6);
\draw[-,thick] (3) to [bend left=15] (5);
\end{tikzpicture}\hspace{3em}
\begin{tikzpicture}[scale=.8]
\node (1) [circle,thick,draw] at (0,0) {1};
\node (2) [circle,thick,draw]  at (-.867,1.3) {2};
\node (3) [circle,thick,draw] at (.867,1.3) {3};
\node (4) [circle,thick,draw] at (0,2.15) {4};
\node (5) [circle,thick,draw] at (-2,3) {5};
\node (6) [circle,thick,draw] at (2,3) {6};
\draw[->,thick] (5) to (4);\draw[->,thick] (5) to (2);\draw[->,thick] (6) to (3);
\draw[->,thick] (4) to (2);
\draw[->,thick] (2) to (1);
\draw[->,thick] (6) to (4);
\draw[->,thick] (4) to (3);
\draw[->,thick] (3) to (1);
\draw[->,thick] (4) to (1);
\end{tikzpicture}
\end{center}
\caption{On the left is the independence graph for the Cambrian lattice of~\Cref{fig:tamari_extremal}.  For comparison, we have reproduced its Galois graph on the right.}
\label{fig:cjcgalois}
\end{figure}

Recall that the \defn{deletion} and \defn{link} of a face $F$ in a simplicial complex $\Delta$ are
\begin{align*}\mathrm{del}_\Delta(F) &= \{G \in \Delta : G \cap F = \emptyset\}, \text{ and}\\\mathrm{lk}_{\Delta}(F) &= \{ G \in \Delta : G \cap F = \emptyset, G \cup F \in \Delta\}.\end{align*}

\begin{proposition}
\label{lem:recursion2}
For $\LL$ a trim lattice, $\gamma(\LL^1)=\mathrm{del}_{\gamma(\LL)}(1)$ and $\gamma(\LL_1) = \mathrm{lk}_{\gamma(\LL)}(1)$.
\end{proposition}
\begin{proof}
The first statement is equivalent to~\Cref{lem:go_up_1}, while the second is a rephrasing of \Cref{lem:down}.
\end{proof}

\begin{proposition}
\label{prop:flag}
For $\LL$ a trim lattice, $\gamma(\LL)$ is flag.
\end{proposition}
\begin{proof}
For $A$ a clique in  the $1$-skeleton of  $\gamma(\LL)$, we must find $y \in \LL$ such that $\down(y)=A$.

 \medskip
\emph{Case I:  $1 \not \in A$.} By induction on length we may find an element $y^1 \in \LL^1$ with $\down_{\LL^1}(y^1)=A$, since $\gamma(\LL^1)=\mathrm{del}_{\gamma(\LL)}(1)$.  If $\down(y^1)=A$, take $y=y^1$; otherwise, $\down_{\LL^1}(y^1)=A \cup \{1\}$ by~\Cref{lem:go_up_1} and by~\Cref{lem:down} we may take $y$ to be the element in $\LL_1$ such that $y^1=j_1 \vee y$.

 \medskip
    \emph{Case II: $1 \in A$.} Then $A \setminus \{1\}$ is a clique in the $1$-skeleton of $\gamma(\LL^1)$, and we can find an element $z$ with $\down(z)=A \setminus\{1\}$ by the previous argument.  But now $y=j_1 \vee z$ has $\down(y)=A$.
\end{proof}

As a consequence of~\Cref{prop:flag}, $\gamma(\LL)$ is determined by its edges.  We define the \defn{independence graph} (by abuse of notation, also denoted $\gamma(\LL)$) to be the 1-skeleton of the independence complex of a trim lattice $\LL$.

\begin{theorem}
For $\LL$ a trim lattice, the complement of the Galois graph $G(\LL)$ is the independence graph $\gamma(\LL)$.
\label{thm:complement}
\end{theorem}

\begin{proof}
\Cref{lem:recursion2,lem:recursion} show that $G(\LL)$ and $\gamma(\LL)$  obey the same induction when $1$ is deleted.  For the remaining edges, note that $\LL=\LL^1 \sqcup \LL_1$ and there is an edge from $k$ to $1$ in $G(\LL)$ if and only if $j_k \in \LL^1$, and an edge between $k$ and $1$ in $\gamma(\LL)$ if and only if $j_k \in \LL_1$.
\end{proof}

Recall that an \defn{independent set} of a graph is a subset of vertices, no two of which are adjacent.  Since $\gamma(\LL)$ and $G(\LL)$ are complementary, we conclude the following characterization of the downward (and upward) labels of a trim lattice.
\begin{corollary}
\label{cor:independent}
For $\LL$ a trim lattice, the independence complex $\gamma(\LL)$ is the set of independent sets of $G(\LL)$, ordered by inclusion.  In particular, the number of elements of a trim lattice $\LL$ depends only on its undirected Galois graph $G(\LL)$.
\end{corollary}

\begin{remark}
\label{rem:same_number}
For a distributive lattice $J(\Q)$, the independent sets of the Galois graph are the antichains in $\Q$; when $\Q$ is a positive root poset, these are called the \defn{nonnesting partitions}.  For a Cambrian lattice $\Camb_c(W)$, the downward labels of elements encode noncrossing partitions.  The setting of trim lattices allows both nonnesting and noncrossing partitions to be viewed as independent sets of two different Galois graphs associated to a finite root system.
For example, \Cref{cor:independent} explains why the lattices in~\Cref{fig:dist_lattice,fig:camb_lattice} have the same number of elements---their undirected Galois graphs are isomorphic (unfortunately, these Galois graphs differ as undirected graphs beyond rank 3).
\end{remark}

We will consider the representation of a trim lattice arising from the independent sets of its Galois graph in a future paper.

  \section{Rowmotion on Semidistributive Lattices}
\label{sec:semidistributive}

In this section, we review the construction of rowmotion on semidistributive lattices due to Barnard~\cite{barnard2016canonical}.  There is a natural sequence of flips when the semidistributive lattice is also extremal; we invoke~\Cref{thm:main_thm} to conclude that this sequence defines rowmotion in slow motion on extremal semidistributive lattices by showing that such lattices are trim.  We relate the canonical join complex of an extremal semidistributive lattice to its Galois graph using the independence complex.

\subsection{Descriptive labellings of semidistributive lattices}

\begin{definition}
A lattice $\LL$ is called \defn{semidistributive} if
\begin{align*}
	x \vee y = x \vee z &\implies x \vee(y \wedge z) = x \vee y\\
	x \wedge y = x \wedge z &\implies x \wedge(y \vee z) = x \wedge y
\end{align*}
for all $x,y,z \in \LL$.
\end{definition}

It follows from the definition of semidistributivity that for any cover $x\lessdot y$ in $\LL$, there is a unique minimal element $j \in \LL$ such that $x\vee j = y$.  This element $j$ is necessarily join-irreducible, and we denote it by \[\gamma_\JJ(x \lessdot y):=\min \{ z \in \LL : x \vee z = y\}.\]  Dually, we write \[\gamma_\MM(x \lessdot y):=\max \{ z \in \LL : z \wedge y = x\}\] for the unique maximal element $m \in \MM$ such that $m \wedge y = x$, necessarily meet-irreducible.  We call $\gamma_\JJ$ and $\gamma_\MM$ \defn{semidistributive labellings}.

In particular, this construction associates a unique meet-irreducible element $m$ to a given join-irreducible element $j$, when applied to its unique cover \defn{$j_*\lessdot j$}.   Similarly, for a given meet-irreducible element $m$, we obtain a unique join-irreducible element $j$ using the cover \defn{$m \lessdot m^*$}.  This defines a bijection between join-irreducibles and meet-irreducibles~\cite[Theorem 2.54]{freese1995free}, and we follow Freese, Jezek, and Nation in denoting this bijection $\kappa: \JJ \to \MM$.  One can compute the bijection $\kappa$ and its inverse as \begin{align*}\kappa(j) &= \max \{ z : z \geq j_* \text{ but } z \not \geq j \} \text{ and} \\ \kappa^{-1}(m) &= \min \{ z : z \leq m^* \text{ but } z \not \leq m \}.\end{align*}

The following proposition regarding the consistency of $\kappa$ is implicit in~\cite{barnard2016canonical}.

  \begin{proposition} \label{equiv} Let $x\lessdot y$ be a cover relation in a semidistributive lattice $\LL$.  Then $\gamma_\MM(x \lessdot y) = \kappa(\gamma_\JJ(x\lessdot y))$.
  \end{proposition}

  \begin{proof}  Write $j=\gamma_\JJ(x\lessdot y)$ and $m=\gamma_\MM(x \lessdot y)$.  We wish to show that $\kappa(j)=m$.

	Let $j_*\lessdot j$ be the unique cover relation below the join-irreducible $j$.  Now $x\leq j_*\vee x < y$, so $j_*\vee x = x$ and $j_*\leq x$.  Therefore $j_*=x\wedge j$.  On the other hand, by definition $j_*=\kappa(j)\wedge j$ and $\kappa(j)$ is the maximum element with this property, so $\kappa(j)\geq x$ but $\kappa(j)\not\geq y$ since $y\geq j$.  Thus, $\kappa(j)\wedge y=x$, which shows that $\kappa(j)\leq m$.

On the other hand, by definition of $m$, we have that $m\wedge y=x$.  Since $y \geq j$, we compute that $m\wedge j = m \wedge (y \wedge j)= (m \wedge y) \wedge j = x \wedge j=j_*$, so that $m\leq \kappa(j)$.

We have therefore shown that $\kappa(j)=m$, as desired.
\end{proof}

\subsection{Rowmotion on semidistributive lattices}
Let $\LL$ be a semidistributive lattice.  As usual, for any element $x\in \LL$, we define the sets of downward and upward labels
	\begin{align*}\down_\JJ(x) &:= \{\gamma_\JJ(y \lessdot x) : \text{ for all } y \text{ such that } y \lessdot x\}, \text{ and}\\
	\up_\JJ(x) &:= \{\gamma_\JJ(x \lessdot y) : \text{ for all } y \text{ such that } x \lessdot y\}.
\end{align*}

A subset $A \subseteq \JJ$ is a \defn{canonical join representation} of $x$ if:
\begin{itemize}
	\item $\bigvee A = x$ but $\bigvee A' < x$ for $A' \subset A$ and
	\item for any $B \subseteq \JJ$ such that $\bigvee B = x$, then for each $b \in B$ there exists $a \in A$ such that $a \leq b$.
\end{itemize}
The \defn{canonical meet representation} is defined analogously.

For a finite lattice $\LL$, semidistributivity is equivalent to the existence of a canonical join and meet representation for each element of $\LL$~\cite[Theorem 2.24]{freese1995free}.  In particular, the canonical join representation of an element is given by its downward labels and its canonical meet representation is $\kappa$ applied to its upward labels.

\begin{theorem}[{\cite[Lemma 3.3]{barnard2016canonical}}]
\label{thm:bar1}
	An element $\x$ of a semidistributive lattice $\LL$ is determined by $\down_\JJ(x)$, and by $\up_\JJ(x)$: \[x=\bigvee \down_\JJ(x) = \bigwedge \kappa\left(\up_\JJ(x)\right).\]
Moreover, $\down_\JJ(x)$ is the canonical join representation for $x$, and $\kappa\left(\up_\JJ(x)\right)$ is its canonical meet representation.
\end{theorem}

The following result gives Condition (G2) on semidistributive lattices.  See also~\cite[Proposition 2.10, Lemma 2.11]{garver2016oriented} for proofs of the corresponding statements on congruence-uniform lattices (which are special cases of semidistributive lattices).

\begin{theorem}[{\cite[Corollary 1.5]{barnard2016canonical}}]
\label{thm:bar2}
	In a semidistributive lattice $\LL$, $A$ is a canonical join representation if and only if $\kappa(A)$ is a canonical meet representation.
\end{theorem}

\begin{corollary}\label{semitwo}  Let $\LL$ be a semidistributive lattice.  Then \[\left\{ \down_\JJ(\x) : \x \in \LL\right\} = \left\{ \up_\JJ(\y) : \y \in \LL\right\}.\]
\end{corollary}
\begin{proof}
This follows from~\Cref{thm:bar1,thm:bar2}---$\down_\JJ(x)$ is a canonical join representation of $x$,
so $\kappa(\down_\JJ(x))$ is a canonical meet representation of some element,
say $y$.  Therefore, $\down_\JJ(x)=\up_\JJ(y)$.
\end{proof}

By~\Cref{thm:bar1} and~\Cref{semitwo}, semidistributive labellings of semidistributive lattices are descriptive, and so \Cref{eq:global_row} defines \defn{rowmotion} by $\row(x):=y,$ where $y$ is the unique element of a semidistributive lattice $\LL$ with $\down_\JJ(x)=\up_\JJ(y)$.

  \subsection{Extremal semidistributive lattices are trim}
Suppose that $\LL$ is a semi-distributive extremal lattice with a fixed maximal-length chain.  As in~\Cref{sec:extremal}, this chain allows us to index the join-irreducible elements and the meet-irreducible elements as $j_1,\ldots, j_n$ and $m_1,\ldots,m_n$ and we identify the join-irreducible elements with their indexing.  In this way, the cover relations of $\LL$ are labelled by \[\gamma(x\lessdot y):=i  \text{ if } \gamma_\JJ(x\lessdot y) = j_i,\]  and \Cref{equiv} shows that this labelling may be equivalently defined by \[\gamma(x\lessdot y):=i \text{ if } \gamma_\MM(x\lessdot y) = m_i.\]

The fixed maximal-length chain gives a linear order of the labels; we conclude that rowmotion can be computed in slow motion by flipping in this order from the following more general statement.
{
\renewcommand{\thetheorem}{\ref{thm:barn}}
  \begin{theorem}
Extremal semidistributive lattices are trim.  (By~\Cref{thm:main_thm}, rowmotion can therefore be computed in slow motion.)
    \end{theorem}
\addtocounter{theorem}{-1}
}

\begin{proof}
	We prove that the labelling $\gamma$ defined above is an interpolating EL-labelling, from which we conclude that $\LL$ is left modular by~\Cref{thm:interpolating}, and hence trim.

    We first show that $\gamma$ is an EL-labelling.  Consider an interval $[x,z]$ in $\LL$.    Let \[i:=\min\{ k : j_k \in \JJ, j_k \leq z, j_k \not \leq x \}\] be the index of the minimal join-irreducible less than $z$ but not less than $x$.  We show that $x \lessdot x\vee j_i$.

	Recall that if $y \in \LL$ has maximal orthogonal pair $(X,Y)$, then we write $y_\JJ=X$ and $y_\MM=Y$.  Since $i$ is the smallest index of any join-irreducible below $z$ but not below $x$, if we write $A=(x\vee j_i)_\JJ \setminus x_\JJ$, then the minimum element of $A$ is $i$.

    Since $(j_i)_\MM \supseteq \{i+1,i+2,\ldots,n\}$, and $(x\vee j_i)_\MM=x_\MM \cap (j_i)_\MM$, if we write $B= x_\MM \setminus (x\vee j_i)_\MM$, then the maximum element of $B$ is $i$.

    Suppose now that there is some $w$ with $x \lessdot w<x\vee j_i$.  The smallest element of $w_\JJ\setminus x_\JJ$ is strictly greater than $i$, while the largest element of $x_\MM \setminus w_\MM$ is at most $i$.  This contradicts \Cref{equiv}.  Therefore $x\lessdot x\vee j_i$.

 By induction, this cover up from $x$ will be the beginning of an increasing chain from $x$ to $z$, and it will lexicographically precede all other chains. Since the label $i$ must appear somewhere on any chain from $x$ to $z$, any increasing chain must begin by $x\lessdot x\vee j_i$, so by induction the increasing chain is unique.

We now show $\gamma$ is interpolating.  Suppose $x \lessdot y \lessdot z$ is not the increasing chain from $x$ to $z$.  Let the increasing chain be
$x=x_0\lessdot x_1 \lessdot \dots \lessdot x_r=z$.  The label
$\gamma(x_0\lessdot x_1)$ must appear on the chain $x \lessdot y \lessdot z$,
and it clearly is not the label of $x\lessdot y$.  Thus it is the label of
$y\lessdot z$, establishing one of the two equalities required for
$\gamma$ to be interpolating.  But because the labelling $\gamma$ also admits a dual description in terms of meet-irreducibles, the other equality follows by duality.
\end{proof}

We conclude the following refinement of~\Cref{thm:rep_are_extremal}.
{
\renewcommand{\thecorollary}{\ref{cor:rep_are_trim}}
\begin{corollary}
If $A$ is representation finite and $\mod A$ has no cycles, then $G(A)$ is the Galois graph of a trim lattice.  The elements of $\LL(G(A))$ are naturally the torsion pairs of $A$, ordered with respect to inclusion of torsion classes (or reverse inclusion of torsion-free classes).
\end{corollary}
\addtocounter{corollary}{-1}
}

\begin{proof}
\Cref{thm:rep_are_extremal} shows that the lattice of torsion pairs is extremal.  It is known that it is semi-distributive \cite[Theorem 4.5]{GM2015}.
By \Cref{thm:barn}, it is therefore trim.
\end{proof}

\begin{remark}
There does not seem to be a simple characterization in the literature of those algebras $A$ that are of finite representation type such that $\mod A$ has no cycles.  If we restrict for simplicity to algebras of the form $kQ/I$, with $I$ an admissible ideal, it is clear that $Q$ must have no oriented cycles, since otherwise there would be a cycle already among the projective indecomposable modules.  However, this restriction together with $A$ being of finite representation type does not guarantee that there are no cycles in $\mod A$, as
\cite[Example 1]{HR93} shows.  On the other hand, a sufficient condition for an algebra of finite representation type to have
no cycles is given in \cite{BL}.
\label{rem:rep_theory}
\end{remark}

\Cref{cor:rep_are_trim} also gives a new (uniform) proof that finite Cambrian lattices of simply-laced type are trim.  Since any finite Cambrian lattice is contained in some finite and simply-laced Cambrian lattice as the sublattice fixed under a group of lattice automorphisms, \cite[Theorem 4]{thomas2006analogue} gives us the following corollary.   

\begin{corollary}
Finite Cambrian lattices are trim.
\end{corollary}

\subsection{Canonical join complexes and Galois graphs}

The \defn{canonical join complex} of a semidistributive lattice $\LL$ is the simplicial complex with vertices the join-irreducible elements $\JJ$ and faces the canonical join representations~\cite[Proposition 2.2]{reading2015noncrossing}.  The canonical join complex of a semidistributive lattice is flag, and hence determined by its edges~\cite[Theorem 1.1]{barnard2016canonical}.  We therefore define the \defn{canonical join graph} $C(\LL)$ to be the 1-skeleton of the canonical join complex of a semidistributive lattice $\LL$.

We observe from~\Cref{thm:bar1} that the canonical join complex of an extremal semidistributive lattice is identical to its independence complex.  Then \Cref{thm:complement} relates the canonical join complex and the Galois graph.

\begin{corollary}
Let $\LL$ be an extremal semidistributive lattice.  Then the complement of its Galois graph $G(\LL)$ is its canonical join graph $C(\LL)$.
\label{thm:semi_dist_complement}
\end{corollary}

We note that it is possible for more than one semidistributive lattice to have the same canonical join graph (for example, the weak order on $\mathfrak{S}_3$ on the right of~\Cref{fig:trim_semi}, the Cambrian lattice of type $B_2$ in~\Cref{fig:b2camb}, $J([2]\times[2])$, and $J(\Phi^+(B_2))$ all have the same canonical join graph).

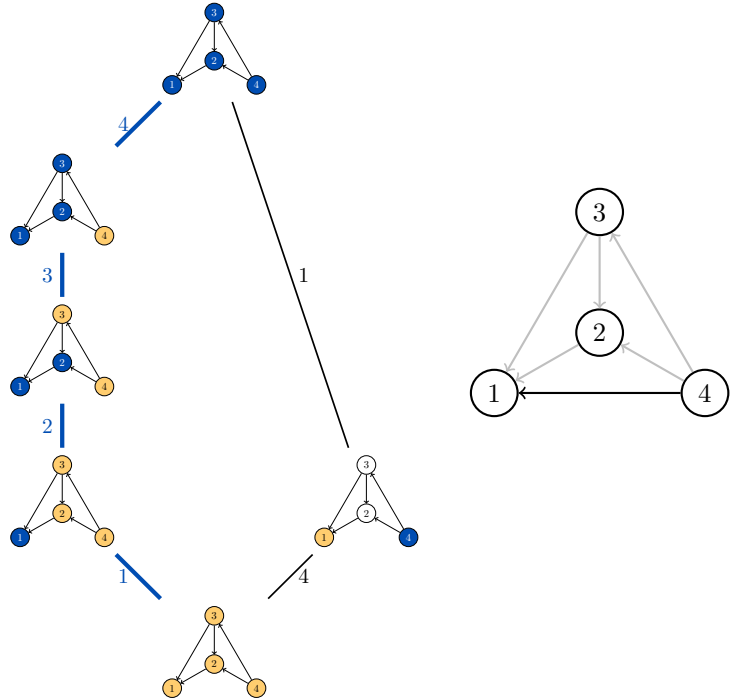
\begin{figure}[htbp]
\raisebox{-0.6\height}{\scalebox{.8}{\begin{tikzpicture}[scale=2.5]
\node (x0) [align=center] at (0,0) {\scalebox{.5}{\begin{tikzpicture}[scale=.8] \node (a) [draw=black,circle,thick,fill=meet] at (-1.732,-1) {$1$}; \node (b) [draw=black,circle,thick,fill=meet] at (1.732,-1) {$4$}; \node (c) [draw=black,circle,thick,fill=meet] at (0,0) {$2$}; \node (d) [draw=black,circle,thick,fill=meet] at (0,2) {$3$}; \draw[<-,thick] (a) to (c); \draw[->,thick] (b) to (c); \draw[->,thick] (b) to (d); \draw[->,thick] (d) to (a); \draw[->,thick] (d) to (c); \end{tikzpicture}}};
\node (x1) [align=center]at (-1,1) {\scalebox{.5}{\begin{tikzpicture}[scale=.8] \node (a) [draw=black,circle,thick,fill=join,text=white] at (-1.732,-1) {$1$}; \node (b) [draw=black,circle,thick,fill=meet] at (1.732,-1) {$4$}; \node (c) [draw=black,circle,thick,fill=meet] at (0,0) {$2$}; \node (d) [draw=black,circle,thick,fill=meet] at (0,2) {$3$}; \draw[<-,thick] (a) to (c); \draw[->,thick] (b) to (c); \draw[->,thick] (b) to (d); \draw[->,thick] (d) to (a); \draw[->,thick] (d) to (c); \end{tikzpicture}}};
\node (j2) [align=center] at (1,1) {\scalebox{.5}{\begin{tikzpicture}[scale=.8] \node (a) [draw=black,circle,thick,fill=meet] at (-1.732,-1) {$1$}; \node (b) [draw=black,circle,thick,fill=join,text=white] at (1.732,-1) {$4$}; \node (c) [draw=black,circle,thick,fill=white] at (0,0) {$2$}; \node (d) [draw=black,circle,thick,fill=white] at (0,2) {$3$}; \draw[<-,thick] (a) to (c); \draw[->,thick] (b) to (c); \draw[->,thick] (b) to (d); \draw[->,thick] (d) to (a); \draw[->,thick] (d) to (c); \end{tikzpicture}}};
\node (x2) [align=center] at (-1,2) {\scalebox{.5}{\begin{tikzpicture}[scale=.8] \node (a) [draw=black,circle,thick,fill=join,text=white] at (-1.732,-1) {$1$}; \node (b) [draw=black,circle,thick,fill=meet] at (1.732,-1) {$4$}; \node (c) [draw=black,circle,thick,fill=join,text=white] at (0,0) {$2$}; \node (d) [draw=black,circle,thick,fill=meet] at (0,2) {$3$}; \draw[<-,thick] (a) to (c); \draw[->,thick] (b) to (c); \draw[->,thick] (b) to (d); \draw[->,thick] (d) to (a); \draw[->,thick] (d) to (c); \end{tikzpicture}}};
\node (x5) [align=center] at (-1,3) {\scalebox{.5}{\begin{tikzpicture}[scale=.8] \node (a) [draw=black,circle,thick,fill=join,text=white] at (-1.732,-1) {$1$}; \node (b) [draw=black,circle,thick,fill=meet] at (1.732,-1) {$4$}; \node (c) [draw=black,circle,thick,fill=join,text=white] at (0,0) {$2$}; \node (d) [draw=black,circle,thick,fill=join,text=white] at (0,2) {$3$}; \draw[<-,thick] (a) to (c); \draw[->,thick] (b) to (c); \draw[->,thick] (b) to (d); \draw[->,thick] (d) to (a); \draw[->,thick] (d) to (c); \end{tikzpicture}}};
\node (x3) [align=center] at (0,4) {\scalebox{.5}{\begin{tikzpicture}[scale=.8] \node (a) [draw=black,circle,thick,fill=join,text=white] at (-1.732,-1) {$1$}; \node (b) [draw=black,circle,thick,fill=join,text=white] at (1.732,-1) {$4$}; \node (c) [draw=black,circle,thick,fill=join,text=white] at (0,0) {$2$}; \node (d) [draw=black,circle,thick,fill=join,text=white] at (0,2) {$3$}; \draw[<-,thick] (a) to (c); \draw[->,thick] (b) to (c); \draw[->,thick] (b) to (d); \draw[->,thick] (d) to (a); \draw[->,thick] (d) to (c); \end{tikzpicture}}};
\draw[-,thick,color=join,line width=2pt] (x0) to node[midway, left] {$1$} (x1) to node[midway, left] {$2$} (x2) to node[midway, left] {$3$} (x5) to node[midway, left] {$4$} (x3);
\draw[-,thick] (x0) to node[midway, right] {$4$} (j2) to node[midway, right] {$1$} (x3);
\end{tikzpicture}}} \hspace{1em}
\raisebox{-0.6\height}{\scalebox{1}{\begin{tikzpicture}[scale=.8] \node (a) [draw=black,circle,thick,fill=white] at (-1.732,-1) {$1$}; \node (b) [draw=black,circle,thick,fill=white] at (1.732,-1) {$4$}; \node (c) [draw=black,circle,thick,fill=white] at (0,0) {$2$}; \node (d) [draw=black,circle,thick,fill=white] at (0,2) {$3$}; \draw[<-,thick,lightgray] (a) to (c); \draw[->,thick,lightgray] (b) to (c); \draw[->,thick,lightgray] (b) to (d); \draw[->,thick,lightgray] (d) to (a); \draw[->,thick,lightgray] (d) to (c); \draw[->,thick] (b) to (a); \end{tikzpicture}}}
\caption{On the left is the extremal semidistributive Cambrian lattice of type $B_2$.  Forgetting orientation, on the right is its canonical join graph (with complementary Galois graph in light gray), which is also the canonical join graph of the non-extremal semidistributive lattice on the right of~\Cref{fig:trim_semi}.}
\label{fig:b2camb}
\end{figure}


\section{History and Examples}
\label{sec:examples}

\subsection{History}
Rowmotion was introduced by Duchet in~\cite{duchet1974hypergraphes}; studied for the Boolean lattice by Brouwer and Schrijver~\cite{brouwer1974period,brouwer1975dual};  and (still for the Boolean lattice) related to matroid theory by Deza and Fukuda~\cite{deza1990loops}.  Cameron and Fon-der-Flaass considered rowmotion on the product of two and then three chains \cite{fon1993orbits,cameron1995orbits}.  Because the orbit structure of rowmotion on Boolean lattices is so wild, much of the effort in the references above is dedicated to understanding which orbit sizes are realizable.

Its study then apparently lay dormant for over a decade until Panyushev resurrected it in the form of a series of conjectures of the orbit structure of rowmotion on the root posets of Lie algebras~\cite{panyushev2009orbits}.  The focus then shifted to finding equivariant bijections to natural combinatorial objects, and Stanley completely characterized the orbit structure of rowmotion on the product of two chains combinatorially~\cite{stanley2009promotion}.  Armstrong, Stump, and Thomas~\cite{armstrong2013uniform} then went on to resolve Panyushev's conjectures using noncrossing partitions, while Striker and Williams  unified and extended various results by relating rowmotion to \textit{jeu-de-taquin} and made linguistic contributions to the theory~\cite{striker2012promotion}.

This popularization of rowmotion led to a swell of related work falling under Propp's heading of ``dynamical algebraic combinatorics.''

  Rush and Shi obtained a beautiful result on rowmotion on minuscule posets \cite{rush2013orbits}, which Rush built on in~\cite{rush2015orbits,rush2016order}.  Propp and Roby returned to the product of two chains and introduced the notion of homomesy~\cite{propp2015homomesy}, which led to research on the distribution of combinatorial statistics across orbits.
Work of Dilks, Pechenik, Striker, and later Vorland connected rowmotion to Thomas and Yong's K-theoretic \textit{jeu-de-taquin}~\cite{dilks2017resonance,dilks2017increasing}.  Striker further extended rowmotion by concentrating on generalizing the concept of toggle~\cite{striker2015toggle,striker2016rowmotion}.  Joseph completed one branch of this program by establishing the relationship between toggles on antichains and toggles on order ideals~\cite{joseph2017antichain}, and other variants of toggles on more varied combinatorial objects appeared in~\cite{schilling2017braid,chan2017expected,einstein2016noncrossing}.

Motivated by Berenstein and Kirillov's action on Gelfand-Tsetlin patterns \cite{kirillov1995groups}, Einstein and Propp considered a piecewise-linear lifting of rowmotion to the order polytope of a poset, and defined an even more general birational generalization \cite{einstein2013combinatorial,einstein2014piecewise}.  Birational rowmotion was studied by Grinberg and Roby \cite{grinberg2015iterative,grinberg2016iterative}; Glick later showed that one of the Grinberg-Roby results was equivalent to an instance of Zamolodchikov periodicity.   Most recently, Galashin and Pylyavskyy generalized the notion of birational rowmotion to strongly connected directed graphs by introducing $R$-systems \cite{galashin2017r}.

\subsection{Boolean lattices} Rowmotion on the lattice of order ideals of \defn{Boolean lattices} was studied in~\cite{duchet1974hypergraphes,brouwer1974period,brouwer1975dual,deza1990loops}; its order is unknown in general: ``the question of which [lengths of orbits] occur for a certain $n$ seems to be untractable''~\cite{brouwer1975dual}.   for $n=1,2,3,4,5,6,7,\ldots$, the maximum orbit sizes are $3,4,5,6,27,\geq 1032,\geq 3791$~\cite{deza1990loops}.

\subsection{Root posets} Rowmotion on \defn{root posets} of Weyl groups was first considered by Panyushev in~\cite{panyushev2009orbits}.  Order ideals of the root poset are called nonnesting partitions~\cite{reiner1997non}, and the order of rowmotion on nonnesting partitions was proven to be $2h$ in~\cite{armstrong2013uniform} by defining an equivariant bijection to the noncrossing partitions under the Kreweras complement (see~\Cref{fig:dist_lattice,fig:camb_lattice}).

Related posets of interest include the poset of generalized nonnesting partitions and the restriction of weak order in the affine symmetric group to simultaneous coprime $(a,b)$-cores, but small computations suggest that the order of rowmotion is not generally well-behaved on these generalizations.

\subsection{Minuscule posets} Rowmotion on \defn{minuscule posets} (arising from highest-weight representations of minuscule weights of Lie algebras) was studied combinatorially for special cases in~\cite{striker2012promotion}, and systematically in~\cite{rush2013orbits,rush2015orbits,rush2016order}.  The order of rowmotion is the Coxeter number $h$: order ideals in a minuscule poset correspond to certain inversion sets of elements in a parabolic quotient, and rowmotion on the order ideals is conjugate to the action of a Coxeter element on the parabolic quotient.  The case of the product of two chains has recieved special attention~\cite{stanley2009promotion,propp2015homomesy,grinberg2015iterative}.

A simple extension---rowmotion on the adjoint crystal (which happens to be a semidistributive lattice)---is \emph{often} well-behaved, for example giving the expected $n$ orbits of size $h+1$ in types $D$ and $E$.

\subsection{Rational Dyck paths}  For $a$ and $b$ coprime, rowmotion on \defn{Dyck paths} in an $a \times b$ rectangle---that is, lattice paths staying above the main diagonal, ordered as a distributive lattice by inclusion---has order $a+b-1$.  This action is related to a rational Kreweras complement and generalizes the example of rowmotion on the type $A$ root poset, since Dyck paths in an $n \times (n+1)$ rectangle recover the nonnesting partitions in the root poset of type $A_n$~\cite{armstrong2013rational,bodnar2015cyclic,bodnar2017rational}.

\subsection{Rational Tamari lattices}
Bergeron and Pr\'{e}ville-Ratelle's \defn{$m$-Tamari lattices} may be defined as certain intervals in the Tamari lattice (and are therefore trim)~\cite{bergeron2012combinatorics,bergeron2011higher,bousquet2012number}.  More generally, one may define a trim \defn{$(a,b)$-Tamari lattice} by rotations on Dyck paths in an $a \times b$ rectangle for $a$ and $b$ coprime.  We {\bf conjecture} that rowmotion on the $(a,b)$-Tamari lattice has order $a+b-1$.  More generally, Pr\'{e}ville-Ratelle and Viennot have defined Tamari lattices on paths staying above a fixed path; as intervals in larger Tamari lattices, these generalizations are all trim~\cite{preville2014extension}. %

\subsection{Grid-Tamari lattices}\label{sec:gridtamari}  McConville's $m \times n$ \defn{Grid-Tamari lattices} are semidistributive lattices~\cite{mcconville2017lattice}, and we {\bf conjecture} that they are trim when $\min(m,n) \leq 3$.  The Galois graph for the (trim) $3 \times 3$ Grid-Tamari lattice is drawn on the left of~\Cref{fig:galois_examples}.  Although the number of vertices of the $m\times n$ Grid-Tamari lattice is equal to the number of standard Young tableaux (SYT) of shape $m \times n$, a definition of cover relations on SYT is not known explicitly~\cite{thomas}.  While the $2 \times n$ Grid-Tamari lattice recovers the Tamari lattice and promotion on $2 \times n$ SYT is in equivariant bijection with the Kreweras complement on noncrossing partitions, we do not know the operation on $m\times n$ SYT corresponding to rowmotion on the $m\times n$ Grid-Tamari lattice (it isn't promotion).

\subsection{Cambrian and $m$-Cambrian lattices}  Reading's \defn{Cambrian lattices} are both trim and semi-distributive---rowmotion on the Cambrian lattices recovers the Kreweras complement on noncrossing partitions, and has order $2h$~\cite{reading2007sortable,reading2007clusters}.  The $m$-Cambrian lattices were introduced in~\cite{stump2015cataland}, and rowmotion there is easily seen to have order $(m+1)h$ using Armstrong's generalization of the Kreweras complement to $m$-noncrossing partitions.  The Galois graph for the $2$-Cambrian lattice of type $A_2$ is drawn on the right of~\Cref{fig:galois_examples}.

\subsection{Torsion pairs of tilted algebras}

Let $A$ be a hereditary Artin algebra.  Choose some torsion pair $(X,Y)$, which by~\Cref{sec:rep_theory} we can also think of as a maximal orthogonal pair for the Galois graph defined on the indecomposable $A$-modules.
Now define a new Galois graph $G_{(X,Y)}$ as follows.  Its vertex set consists of the indecomposables in $X$ together with the indecomposables of $Y$.  Restricted to each of $X$ and $Y$, the graph $G_{(X,Y)}$ agrees with $G$.  However, $G_{(X,Y)}$ does not have any of the arrows from $Y$ to $X$ which appear in $G$.  Rather, for each $M$ in $X$ and $N$ in $Y$, we add an arrow from $M$ to $N$ if and only if $\Hom(N, \tau M)\ne 0$.  Since we are in the hereditary setting, $\tau$ can be understood combinatorially: it sends $M$ to 0 if $M$ is projective, and otherwise acts on
dimension vectors by $s_n\cdots s_1$, where the simple reflections are numbered in an order consistent with the order given by the transitive closure of the Galois graph.  The resulting graph $G_{(X,Y)}$ is the Galois graph for the tilted algebra obtained by tilting $A$ at $(X,Y)$, in the sense of \cite[Corollary I.2.2]{HRS}, and the corresponding lattice is therefore trim by~\Cref{cor:rep_are_trim}.  Although the Hasse diagrams of the torsion pairs of these tilted algebras are regular, rowmotion does not appear to have predictable order: already in $A_4$ there are tilted algebras for which the order of rowmotion is $90$.



\subsection{biCambrian lattices}\label{sec:bicambrian}  Barnard and Reading introduced the \defn{biCambrian lattices} in~\cite{reading2015coxeter} as the restriction of weak order to those elements whose canoncial join representation uses only join-irreducible $c$- or $c^{-1}$-sortable elements, where $c$ is a bipartite Coxeter element.  For the coincidental types $A,B,H_3,$ and $I_2(m)$, Barnard and Reading's doubled root posets coincide with certain minuscule posets, and the number of bisortable elements is the same as the number of order ideals in those minusucle posets.   We {\bf conjecture} that the orbit structure of rowmotion on the (semidistributive) biCambrian lattices of those types coincides with the orbit structure of rowmotion on the corresponding minuscule posets.

\subsection{Weak Order}  \defn{Weak order} on a finite Coxeter group is a semidistributive lattice, but the order of rowmotion appears unpredictable---for the symmetric groups $\mathfrak{S}_n$ for $n=1,2,3,4,5,6,7,\ldots$, the maximum orbit sizes are $1$, $2$, $4$, $12$, $20$, $128$, $412,\ldots$.

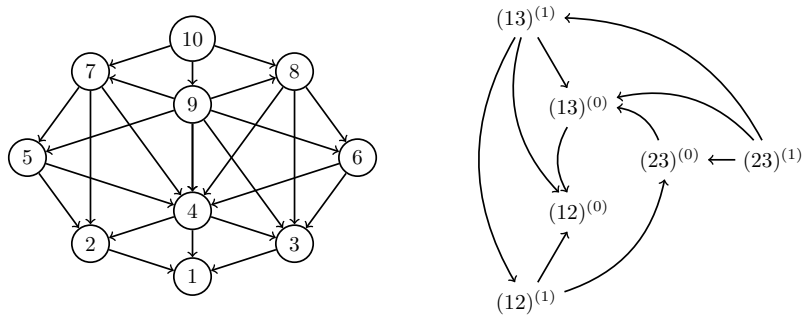
\begin{figure}[htbp]
\raisebox{-0.6\height}{\scalebox{0.8}{\begin{tikzpicture}[scale=1.5] \node (a) [draw=black,circle,thick,fill=white] at (1.80902,0) {$6$}; \node (b) [draw=black,circle,thick,fill=white] at (1.11803,-0.951057) {$3$}; \node (c) [draw=black,circle,thick,fill=white] at (0,-.587785) {$4$}; \node (d) [draw=black,circle,thick,fill=white] at (0,0.587785) {$9$}; \node (e) [draw=black,circle,thick,fill=white] at (1.11803,.951057) {$8$}; \node (f) [draw=black,circle,thick,fill=white] at (-1.11803,.951057) {$7$};  \node (g) [draw=black,circle,thick,fill=white] at (-1.80902,0) {$5$}; \node (h) [draw=black,circle,thick,fill=white] at (-1.11803,-.951057) {$2$}; \node (i) [draw=black,circle,thick,fill=white] at (0,1.31433) {$10$}; \node (j) [draw=black,circle,thick,fill=white] at (0,-1.31433) {$1$}; \draw[->,thick] (b) to (j);  \draw[->,thick] (h) to (j);  \draw[->,thick] (c) to (h);  \draw[->,thick] (c) to (b);  \draw[->,thick] (c) to (j);  \draw[->,thick] (a) to (b);  \draw[->,thick] (a) to (c);  \draw[->,thick] (d) to (c);  \draw[->,thick] (d) to (b);  \draw[->,thick] (d) to (a);  \draw[->,thick] (d) to (c);  \draw[->,thick] (g) to (c);  \draw[->,thick] (d) to (e);  \draw[->,thick] (e) to (c);  \draw[->,thick] (e) to (b); \draw[->,thick] (e) to (a);  \draw[->,thick] (g) to (h);  \draw[->,thick] (f) to (g);  \draw[->,thick] (f) to (h);  \draw[->,thick] (f) to (c);  \draw[->,thick] (d) to (f);  \draw[->,thick] (d) to (g);  \draw[->,thick] (i) to (f);  \draw[->,thick] (i) to (d);  \draw[->,thick] (i) to (e);  \end{tikzpicture}}}\hspace{3em} \raisebox{-0.6\height}{\scalebox{0.8}{\begin{tikzpicture}{scale=1.4}
\node (1) at (1,0) {$(23)^{(0)}$};
\node (2) at (2.7,0) {$(23)^{(1)}$};
\node (3) at (-.5,.866) {$(13)^{(0)}$};
\node (4) at (-1.359,2.3541) {$(13)^{(1)}$};
\node (5) at (-.5,-.866) {$(12)^{(0)}$};
\node (6) at (-1.359,-2.3541) {$(12)^{(1)}$};
\draw[->,thick] (2) to (1);
\draw[->,thick] (4) to (3);
\draw[->,thick] (6) to (5);
\draw[->,thick] (2) to [bend right] (3);
\draw[->,thick] (2) to [bend right] (4);
\draw[->,thick] (1) to [bend right] (3);
\draw[->,thick] (3) to [bend right] (5);
\draw[->,thick] (4) to [bend right] (5);
\draw[->,thick] (4) to [bend right] (6);
\draw[->,thick] (6) to [bend right] (1);
\end{tikzpicture}}}
\caption{On the left is the Galois graph for the $3 \times 3$ Grid-Tamari lattice.  On the right is the Galois graph for a $2$-Cambrian lattice of type $A_2$.}
\label{fig:galois_examples}
\end{figure}

\section*{Acknowledgements}
The first author was partially supported by an NSERC Discovery Grant and the Canada Research Chairs program.

We thank Thomas McConville for his observations regarding~\Cref{sec:gridtamari}.  The second author thanks Emily Barnard for a conversation at FSPAC 2015 regarding~\Cref{sec:bicambrian}.  We thank Darij Grinberg and Nathan Reading for kindly pointing out typos in an earlier version of this manuscript.

\bibliographystyle{amsalpha}
\bibliography{trim}
\end{document}